\documentclass[12pt, leqno]{amsart}
\usepackage[OT2,T1]{fontenc}
\DeclareSymbolFont{cyrletters}{OT2}{wncyr}{m}{n}
\DeclareMathSymbol{\Sha}{\mathalpha}{cyrletters}{"58}

\usepackage{indentfirst}
\usepackage{amstext}
\usepackage{amsopn}
\usepackage{amsfonts}
\usepackage{amsmath}
\usepackage{latexsym}
\usepackage{amscd}
\usepackage{amssymb}
\usepackage{amsmath}
\usepackage[all,cmtip]{xy}
\usepackage{leftidx}
\usepackage{graphicx}
\usepackage{tikz}
\usepackage{ulem}

=\oddsidemargin \font\teneufm=eufm10 \font\seveneufm=eufm7
\font\fiveeufm=eufm5
\newfam\eufmfam
\textfont\eufmfam=\teneufm \scriptfont\eufmfam=\seveneufm
\scriptscriptfont\eufmfam=\fiveeufm

\let\goth\mathfrak
\def\cA{\mathcal A}
\def\cC{\mathcal C}

\def\cO{\mathcal O}

\def\cLie{\mathcal Lie}
\def\cE{\mathcal E}
\def\cF{\mathcal F}
\def\cI{\mathcal I}

\def\cN{\mathcal N}
\def\cX{\mathcal X}
\def\cY{\mathcal Y}

\def\gm{\goth m}

\def\GG{\mathbb{G}}

\def\FF{\mathbb{F}}
\def\WW{\mathbf{W}}
\def\VV{\mathbf{V}}
\def\bV{\mathbb{V}}

\def\gg{\goth g}

\def\gY{\goth Y}
\def\gX{\goth X}

\def\gg{\goth g}

\def\1{\mbox{\bf 1}}

\def\rad{\mathrm{rad}}

 \DeclareMathOperator{\Hom}{Hom}
\DeclareMathOperator{\Aut}{Aut}

\DeclareMathOperator{\Isom}{Isom}

\DeclareMathOperator{\Sup}{Sup}

\DeclareMathOperator{\GL}{\rm GL}
\DeclareMathOperator{\SL}{\rm SL}

\newcommand{\incl}[1][r]
{\ar@<-0.2pc>@{^(-}[#1] \ar@<+0.2pc>@{-}[#1]}

\newtheorem{stheorem}{Theorem}[section]%Theorems et al in italic font.
\newtheorem{sclaim}[stheorem]{Claim}

\newtheorem{scorollary}[stheorem]{Corollary}
\newtheorem{slemma}[stheorem]{Lemma}
\newtheorem{sproposition}[stheorem]{Proposition}
\newtheorem{sremark}[stheorem]{Remark}
\newtheorem{sremarks}[stheorem]{Remarks}
\newtheorem{sexample}[stheorem]{Example}

\theoremstyle{definition}%Theorems et al in roman font.

\numberwithin{equation}{section}

\def\ZZ{\mathbb{Z}}

\def\PP{\mathbb{P}}

\def\Bun{\text{\rm Bun}}

\def\cO{\mathcal{O}}

\def\ol{\overline}

\def\e{\mathrm{e}}

\def\Lie{\mathop{\rm Lie}\nolimits}

\def\2int{\mathop{2\int}\nolimits}

\def\Spec{\mathop{\rm Spec}\nolimits}
\def\Sup{\mathop{\rm Sup}\nolimits}
\def\Lie{\mathop{\rm Lie}\nolimits}

\def\Hom{\mathop{\rm Hom}\nolimits}

\def\Inf{\mathop{\rm Inf}\nolimits}

\def\Pic{\mathop{\rm Pic}\nolimits}

\def\Tors{\mathop{\rm Tors}\nolimits}

\def\Aut{\text{\rm{Aut}}}

\def\Isom{\mathop{\rm Isom}\nolimits}

\def\resp.{\mathop{\rm resp.}\nolimits}
\def\limproj{\mathop{\oalign{lim\cr
\hidewidth$\longleftarrow$\hidewidth\cr}}}

\def\lgr{\longrightarrow}

\font\math=cmmi10
\def\varpi{\hbox{\math\char'44}}

\def\simlgr{\buildrel\sim\over\lgr}

\def\pa{\S\kern.15em }

\def\un{\uppercase\expandafter{\romannumeral 1}}
\def\deux{\uppercase\expandafter{\romannumeral 2}}
\def\trois{\uppercase\expandafter{\romannumeral 3}}
\def\quatre{\uppercase\expandafter{\romannumeral 4}}
\def\cinq{\uppercase\expandafter{\romannumeral 5}}
\def\six{\uppercase\expandafter{\romannumeral 6}}

\def\gg{\goth g}
\def\gsl{\goth sl}

\def\hfl#1#2#3{\smash{\mathop{\hbox to#3{\rightarrowfill}}\limits
^{\scriptstyle#1}_{\scriptstyle#2}}}
\def\gfl#1#2#3{\smash{\mathop{\hbox to#3{\leftarrowfill}}\limits
^{\scriptstyle#1}_{\scriptstyle#2}}}

\title[Local triviality]{Local triviality for $G$-torsors}

\author{P. Gille}\address{UMR 5208
Institut Camille Jordan - Universit\'e Claude Bernard Lyon 1
43 boulevard du 11 novembre 1918
69622 Villeurbanne cedex - France 
}

\email{gille@math.univ-lyon1.fr}

\author{R. Parimala}\address{Departement  of Mathematics and Computer Science,
MSC W401, 400 Dowman Dr. Emory University
Atlanta, GA 30322
USA  
}
\email{parimala.raman@emory.edu}

\author{V. Suresh}\address{Departement  of Mathematics and Computer Science,
MSC W401, 400 Dowman Dr. Emory University
Atlanta, GA 30322
USA  
}
\email{suresh@mathcs.emory.edu}

\date{\today}

\begin{document}

 \begin{abstract} Let $C \to \Spec(R)$ be a relative proper
 flat curve over a henselian base. Let $G$ be a reductive $C$--group scheme.
 Under mild technical assumptions, we show that a $G$--torsor over $C$ which is 
 trivial on the closed fiber of $C$ is locally trivial for the Zariski topology. 

\smallskip


\end{abstract}

\maketitle

% {\small \tableofcontents }

\bigskip

\section{Introduction}\label{section_intro}

The purpose of the paper is to study local triviality 
for $G$--torsors over a 
a {\it relative curve} over an affine base $\Spec(R)$, that is a 
flat proper morphism $C \to \Spec(R)$ of finite presentation 
whose fibers have dimension  $\leq 1$. 
% (it is the schematical case of \cite[Tag 0D4Z]{St})
We deal here with semisimple $C$-group schemes 
which are not necessarily extended from $R$.
Our main result can be stated as follows.

\begin{stheorem}\label{thm_main} Let $R$ 
be a heneselian local noetherian ring with residue field $\kappa$.
 Let $f: C \to \Spec(R)$ be a relative curve of relative dimension $1$
 and denote by $C^{sm}$ the smooth locus of $f$.
We assume that   one of the following holds:

\smallskip

\noindent (I) $C^{sm}_\kappa$ is dense in $C_\kappa$;

\smallskip

\noindent (II) $R$ is a DVR and $C$ is integral regular.

\smallskip

Let $G$ be a semisimple $C$-group scheme
and denote by $q: G^{sc} \to G$ its simply connected covering.  
 We assume that the fundamental group $\mu/C= \ker(q)$ of $G$ is \'etale over $C$.
 Let $E$, $E'$  be  two $G$-torsors over $C$ 
such that 
$E \times_{C} C_\kappa$ is isomorphic
to $E' \times_C C_\kappa$. Then $E$ and $E'$ are  locally isomorphic
for the Zariski topology.

\end{stheorem}

%\begin{sremark}{\rm  
%Note that $\cO_S \to f_*\cO_C$
%is universally an isomorphism in case (I)
%and that case (II) includes the case when $C$ is normal,  
%$\cO_S \to f_*\cO_X$ is an isomorphism and the $g.c.d.\!$ of the geometric multiplicities
%of the irreducible components of $C_\kappa$ is prime
%to the characteristic exponent of $k$. For more details, see Examples \ref{rem_flat}.
%}
%\end{sremark}

\noindent We recall that relative dimension $1$
means that all nonempty fibers $C_s$ are equidimensional of dimension $1$ \cite[Tag 02NJ]{St}.
A related  result is that of Drinfeld and Simpson \cite[th.\ 2]{DS}.
In the case $G$ is semisimple split  and 
$R$ is strictly henselian  they showed  in the smooth case (I) that  
a $G$--torsor over $C$ is locally trivial for the Zariski topology;
according to one of the referees, inspection of the proof
shows this extends to the case of a henselian ring $R$ in the case $C(R) \not = \emptyset $.
Drinfeld-Simpson's result has been generalized recently by 
Belkale-Fakhruddin to a wider setting \cite{BF1, BF2}. 
 We provide a variant in Theorem \ref{thm_ref}
in the case of a henselian base and without any splitting assumptions;
more precisely we do not assume that $G$ admits a proper parabolic subgroup.

One important difference is that we only require that  the ring $R$ is
henselian.  We consider Zariski triviality on $C$
with respect to  henselian (or Nisnevich) topology on the base  while
Drinfeld-Simpson deal with Zariski triviality on $C$ with respect to  the 
 \'etale topology on the base.
We stated the semisimple case but the result extends 
in the reductive case by combining with the case of tori, see 
Theorem \ref{thm_red}.
We denote by $F$ the function field
of $C$; in the case of a DVR, the main result 
leads to  new cases of a local-global principle 
for $G_{F}$--torsors (Corollary \ref{cor_ref}).
More precisely if $C$ is smooth over $R$
with geometrically connected fibers and $G$ is
a semisimple group over $C$ (whose fundamental group is \'etale),
then a torsor under $G$ over $F$ which is trivial at all completions 
of $F$ at discrete valuations of $F$ is trivial.

%%%%%%%%%%%%%%%%%%%%

Let us review the contents of the paper.
The toral case is quite different from the semisimple one
since it works in higher  dimensions;  it  is treated in section \ref{section_tori}
by means of the proper base change theorem.
The section \ref{section_KT} deals with generation by  one-parameter subgroups, namely
the Kneser-Tits problem. 
Section \ref{section_moduli} extends Sorger's construction of the moduli stack 
of $G$--bundles \cite{So} and discusses in detail its tangent bundle.
The next  section \ref{section_unif} recollects facts on patching for $G$--torsors
and provides the main technical statement,
namely the parametrization of the deformations of a given torsor in the 
henselian case in the  presence of isotropy (Proposition \ref{prop_section});
this refines Heinloth's uniformization \cite{He1}).
Section \ref{section_result} explains why this intermediate statement
is enough for establishing the fact that deformations of a given torsor
(in the henselian case) are locally trivial for the Zariski topology.
One important point is that we can get rid of the isotropy assumptions. 
Finally section \ref{section_gather} provides a general 
theorem for reductive groups. 
We include at the end a short appendix \ref{section_appendix} gathering
facts on smoothness for morphisms of algebraic stacks.

\medskip

\noindent{\bf Acknowledgements.} 
 We thank Laurent Moret-Bailly
for the extension of the toral case beyond curves,  a strengthened version 
of Lemma  \ref{lem_qs}  and several  suggestions. 
We thank   Jean-Louis Colliot-Th\'el\`ene for communicating to 
us his method to deal with tori.  We thank Olivier Benoist
for raising a question answered by Theorem \ref{thm_ref}.

Finally we thank Lie Fu,  Ofer Gabber,  Jochen  Heinloth and Anastasia Stavrova 
for useful discussions.

\smallskip

The first author is supported  by the project ANR Geolie, ANR-15-CE
40-0012, (The French National Research Agency).
The  second and third   authors are   partially supported by National
Science Foundation grants DMS-1463882 and DMS-1801951.

\medskip

\noindent{\bf Conventions and Notations.}
We use mainly  the terminology and notations of Grothendieck-Dieudonn\'e \cite[\S 9.4  and 9.6]{EGA1}
 which agree with that  of Demazure-Grothendieck used in \cite[Exp. I.4]{SGA3}

(a) Let $S$ be a scheme and let $\cE$ be a quasi-coherent sheaf over $S$.
 For each morphism  $f:T \to S$, 
we denote by $\cE_{T}=f^*(\cE)$ the inverse image of $\cE$ 
by the morphism $f$.
 We denote by $\VV(\cE)$ the affine $S$--scheme defined by 
$\VV(\cE)=\Spec\bigl( \mathrm{Sym}^\bullet(\cE)\bigr)$;
it is affine  over $S$ and 
represents the $S$--functor $Y \mapsto \Hom_{\cO_Y}(\cE_{Y}, \cO_Y)$ 
\cite[9.4.9]{EGA1}. 
%This construction generalizes to algebraic spaces. 

\smallskip

(b) We assume now that $\cE$ is locally free and denote by $\cE^\vee$ its dual.
In this case the affine $S$--scheme $\VV(\cE)$ is  of finite presentation 
(ibid, 9.4.11); also
the $S$--functor $Y \mapsto H^0(Y, \cE_{Y})= 
\Hom_{\cO_Y}(\cO_Y, \cE_{Y} )$ 
is representable by the  affine $S$--scheme $\VV(\cE^\vee)$
which is also denoted by  $\WW(\cE)$  \cite[I.4.6]{SGA3}.
 %Once again this construction generalizes to algebraic spaces. 

It applies to the locally free coherent sheaf
${\cE}nd(\cE) = \cE^\vee \otimes_{\cO_S} \cE$ 
 over $S$ so that we can consider
the affine $S$--scheme $\VV\bigl({\cE}nd(\cE)\bigr)$
which is an $S$--functor in associative commutative and unital algebras
\cite[9.6.2]{EGA1}.
Now we consider the $S$--functor $Y \mapsto \Aut_{\cO_Y}(\cE_{Y})$.
It is representable by an open $S$--subscheme of $\VV\bigl({\cE}nd(\cE)\bigr)$
which is denoted by $\GL(\cE)$ ({\it loc. cit.}, 9.6.4).

\smallskip

(c) We denote by $\ZZ[\epsilon]= \ZZ[x]/x^2$ the ring of dual integers and 
by $S[\epsilon]=S \times_\ZZ \ZZ[\epsilon]$.
If $G/S$ is an $S$--group space (i.e. an algebraic space in groups, called
group algebraic space over $S$ in 
\cite[Tag 043H]{St}) we
denote by $\Lie(G)$ the $S$--functor
defined by $\Lie(G)(T)= \ker\bigl( G(T[\epsilon]) \to G(T) \bigr)$.
This $S$-functor is a functor in Lie $\cO_S$-algebras, see 
\cite[II.4.1]{SGA3}  or \cite[II.4.4]{DG}. More facts are collected
in Appendix \ref{app_lie}.

\smallskip

\smallskip

(d) If $G/S$ is an affine smooth $S$--group scheme, we denote by $\Tors_G(S)$
the groupoid of  (right)  $G$--torsors over $S$ and by $H^1(S,G)$ the 
set of isomorphism classes of $G$--torsors (locally trivial for the \'etale topology), 
we have a classifying map $\Tors_G(S) \to  H^1(S,G)$, $E \mapsto [E]$.

\medskip

\bigskip

\section{The case of tori}\label{section_tori}
In this section we prove variants of Theorem \ref{thm_main} for nice tori over proper schemes
over local henselian noetherian rings (not necessarily relative curves). The main statement is Theorem \ref{thm_torus}
 which is used in Theorem \ref{thm_ref} on reductive group schemes.

\begin{slemma}\label{lem_MB} Let $X$ be a scheme and 
let $T$ be an $X$--torus.
 Assume that $T$ is split by a finite \'etale cover of degree
$d$. Then $d H^1(X,T) \subseteq H^1_{Zar}(X,T)$.
\end{slemma}

\begin{proof} Let 
$f : Y  \to X$ be a finite \'etale cover of degree $d$
which splits $T$, that is $T_Y \cong \GG_{m,Y}^r$. 
According to \cite[0.4]{CTS}, we have a norm
map $f_*: H^1(Y,T) \to H^1(X,T)$ such that the composite
$H^1(X,T) \xrightarrow{f^*} H^1(Y,T) \xrightarrow{f_*}  H^1(X,T)$
is the multiplication by $d$.  We claim that $f_*\bigl( H^1(Y,T) \bigr)
\subseteq H^1_{Zar}(X,T)$.
Let $x \in X$. Then $V_x= {\Spec(\cO_{X,x}) \times_X Y}$ is a semilocal 
scheme so that $\Pic(V_x)=0$ \cite[Chap. 2, Sect. 5, No. 3, Proposition 5]{BAC}. 
Since $T$ is split over $V_x$,  $H^1(V_x,T)=0$. Since the construction of the norm commutes 
with base change,
we get that $\Bigl( f_*\bigl( H^1(Y,T) \bigr) \Bigr)_{\cO_{X,x}}=0$.
In particular we have $d  H^1(X,T)  \subseteq H^1_{Zar}(X,T)$.
\end{proof}

\newpage

 \begin{sproposition}\label{prop_MB} Let $R$ be a henselian 
 local ring of 
residue field $\kappa$. We denote by $p$ 
the characteristic exponent of $\kappa$
and let $l$ be a prime number distinct
from  $p$. 
Let $X$ be a proper $R$-scheme
and let $T$ be an $X$--torus.

\smallskip

\noindent (1)  For each $i \geq 0$, 
$\ker\bigl( H^i(X, T) \to H^i(X_\kappa,T) \bigr)$
is $l$--divisible.

\smallskip

\noindent (2) 
The kernel
$\ker\bigl( H^0(X, T) \to H^0(X_\kappa,T) \bigr)$
is {\it uniquely} $l$--divisible.

\smallskip

\noindent (3) We assume that  $T$ is locally isotrivial, that is
 there exists an open cover $(U_i)_{i=1,\dots,n}$ of $X$
and finite \'etale covers $f_i: V_i \to U_i$ such that 
$T_{V_i}$ is split for $i=1,\dots,n$.
Then there exists $r \geq 0$ such that
$p^r \ker\bigl( H^1(X, T) \to H^1(X_\kappa,T) \bigr) 
\subseteq H^1_{Zar}(X,T)$. 

\end{sproposition}

\begin{proof}
(1) We consider the 
exact sequence of \'etale $X$--sheaves  
$1 \to {_lT} \to T \xrightarrow{\times l} T \to 1$
which generalizes the Kummer sequence.
It gives rise to the 
following commutative diagram 

\[\xymatrix@1{
 H^i(X, {_lT}) \ar[r] \ar[d]^{\beta_1} &
H^i(X, T)
\ar[r]^{\times l} \ar[d] &   H^i(X, T)  \ar[d]
\ar[r] & H^{i+1}(X, {_lT})  \ar[d]^{\beta_2}
\\
 H^i(X_\kappa,{_lT}) \ar[r]  &
H^i(X_\kappa, T)
\ar[r]^{\times l}  &   H^i(X_\kappa, T) 
\ar[r] & H^{i+1}(X_\kappa, {_lT}) .
}\]
The proper base change theorem  \cite[XII.5.5.(iii)]{SGA4} 
shows that $\beta_1$, $\beta_2$ are isomorphisms.
By  diagram chase, we conclude that 
$\ker\bigl( H^i(X, T) \to H^i(X_\kappa, T) \bigr)$
is $l$--divisible.

\smallskip

\noindent (2) If $i=0$, we can complete the left-hand side of the diagram with $0$.
By  diagram chase it follows that 
$\ker\bigl( H^0(X, T) \to H^0(X_\kappa, T) \bigr)$
is uniquely $l$--divisible.

\smallskip

\noindent (3) Let  $(U_i)_{i=1,\dots,n}$ be an open cover of $X$
and finite \'etale covers $f_i: V_i \to U_i$ such that 
$T_{V_i}$ is split for $i=1,\dots,n$. Let $d$ be the l.c.m.\!\! of the degrees
of the $f_i$'s. We write $d=p^r e$ with $(e,p)=1$.
Assertion (1) shows that
$\ker\bigl( H^1(X, T) \to H^1(X_\kappa,T) \bigr)$ is $e$-divisible
so that $\ker\bigl( H^1(X, T) \to H^1(X_\kappa,T) \bigr) \subseteq e H^1(X,T)$.
Lemma \ref{lem_MB} shows that  $d H^1(X,T) \subseteq H^1_{Zar}(X,T)$
which permits to conclude that we have the inclusion 
$p^r \ker\bigl( H^1(X, T) \to H^1(X_\kappa,T) \bigr) 
\subseteq H^1_{Zar}(X,T)$. 
 \end{proof}

\begin{sremark}{\rm Proposition \ref{prop_MB}.(1) fails completely 
for $l=p$ if the residue field $\kappa$ is of characteristic $p>0$ and
already for $\GG_m$. For example we take $R=\FF_p[[t]]$ (or $\ZZ_p$).
Then $\ker( R^\times \to \FF_p^\times)$ admits $\FF_p$ as quotient
so is not $p$--divisible. For the $H^1$, we consider a smooth  elliptic curve 
$E$ over $k$ having a $k$--point $0$. 
 We have $\Pic(E_{k[[t]]}) \simlgr  \Pic(E_{k((t))}) = E\bigl(k((t))\bigr) \oplus \ZZ =   E\bigl(k[[t]]\bigr) \oplus \ZZ$.
It follows that  \break $\ker\bigl( \Pic(E_{k[[t]]}) \to \Pic(E) \bigr)
\simlgr \ker\Bigl( E\bigl(k[[t]]\bigr) \to E\bigl(k\bigr) \Bigr)$ 
so that admits a quotient isomorphic to $\FF_p= \Lie(E)(\FF_p)$.
Therefore  $\ker\bigl( \Pic(E_{k[[t]]}) \to \Pic(E) \bigr)$ is not $p$-divisible.
}
\end{sremark}
 
 \newpage
 
\begin{stheorem}\label{thm_torus}
Let $R$ be a local henselian noetherian ring with residue 
field $k$ and $p = \mathrm{char}(k)$.
Let  $X \to \Spec(R)$ be a proper scheme and $T$ an $X$-torus.
   Suppose $T$  quasi-splits after 
a finite \'etale extension $X'/X$ of degree prime to $p$, 
 that is \break
 $T\times_X X' \cong  R_{X''/X'}(\GG_m)$ where
 $X''\to X'$ is a finite \'etale cover. 
Then \break $\ker\bigl( H^1(X, T) \to H^1(X_\kappa,T) \bigr) 
\subseteq H^1_{Zar}(X,T)$.
\end{stheorem}

\begin{proof} The theorem 90 of Hilbert-Grothendieck shows
that $H^1_{Zar}(X',T)= H^1(X',T_{X'})$. 
By corestriction-restriction it follows that 
$[X':X] H^1(X,T) \subseteq  H^1_{Zar}(X,T)$.
In particular we have 
 $$
 [X':X] \ker\bigl( H^1(X, T) \to H^1(X_\kappa,T) \bigr) 
\subseteq H^1_{Zar}(X,T).
$$
Since $[X':X]$ is prime to $p$, 
 Proposition \ref{prop_MB} yields that 
 $\ker\bigl( H^1(X, T) \to H^1(X_\kappa,T) \bigr) 
\subseteq H^1_{Zar}(X,T)$.
\end{proof}

\begin{scorollary} \label{cor2_torus} 
Suppose $\mathrm{char}(k)= 0$. 
Then  $ \ker\bigl( H^1(X, T) \to H^1(X_\kappa,T) \bigr) 
\subseteq H^1_{Zar}(X,T)$.
\end{scorollary}

\section{Infinitesimal Kneser-Tits problem}\label{section_KT}

For a semisimple group scheme
$G$ defined over a semilocal ring $R$, we are interested 
in the quotient of $G(R)$ by the subgroup  generated by the unipotent radicals
of parabolic subgroups of $G$.

\subsection{Strictly proper parabolic subgroups}

Let $G$ be a reductive group over an algebraically closed field $k$.
We have a decomposition of its adjoint group 
$$
G_{ad} \simlgr  G_{1} \times \dots \times G_{r_s}, 
$$
where the $G_i$'s are adjoint simple $k$-groups. 
Let $Q$ be a parabolic subgroup of $G$. Then  $Q/C(G)$ is 
a   parabolic subgroup  of  
 $G_{ad}$  and decomposes as  a product of parabolic subgroups
$\prod_i Q_i$.  We say that $Q$ is a {\it strictly proper} parabolic subgroup of 
$G$ if   $Q_i \subsetneq G_i$ for $i=1,\dots,r_s$.

Let $S$ be a scheme and  let $G$ be a reductive $S$--group scheme.
Let $P$ be a parabolic subgroup scheme of $G$. We say that 
$P$ is {\it strictly proper} is 
for each point $s \in S$, $P_{\ol{\kappa(s)}}$ 
is a strictly proper parabolic subgroup of $G_{\ol{\kappa(s)}}$.

This definition is stable under base change and is local for the fppf topology.

\subsection{Last term of Demazure's filtration}\label{subsec_last} 
We continue with the $S$--parabolic subgroup $P$ scheme of $G$
assumed strictly proper.
%Our goal is to construct  a nice vector $S$--subgroup scheme
%of $P$. 
We consider   
Demazure's filtration of the unipotent radical $U=\rad_u(P)$ 
$$
U =U_0 \supset U_1 \dots \supset U_n \supsetneq 0
$$
by vector subgroup schemes which are characteristic in $P$ \cite[XXVI.2.1]{SGA3}.
The last $U_n$ is central in $U$.
If the Cartan-Killing type of $G$ is  constant and connected
and if $P$ is of constant type, the last term  $U_n$ is the right object for our purpose. 
This construction does not behave well under  products
and we need to refine  it; it is enough to deal with the adjoint case
since the unipotent radical of $P$ and $P/C(G)$ are isomorphic.

 \begin{slemma} \label{lem_last}  
 Assume that $G$ is adjoint. The group scheme 
 $U=\rad_u(P)$  admits  a unique closed $S$--subgroup scheme $U_{last}$ satisfying the following property:
 
 \smallskip
 
 {\rm For each $S$--scheme $Y$ such that $G_Y$, $P_Y$ are
   of constant type and 
 such that  \break $G \times_S Y \, {\buildrel \phi \over \cong} \, G_1 \times_Y  G_2 \dots \times_Y G_c$ where $G_i$ is an adjoint semisimple
 $Y$--scheme  whose absolute Cartan-Killing type is  connected, then 
 $\phi(U_{last,Y})= U_{1,last} \times_Y \dots \times_Y U_{c,last}$ where $U_{i,last}$ is the last term 
 of  Demazure's filtration of $P_Y \cap G_i$ for $i=1,\dots,c$.
 }
 \smallskip
 
 Furthermore, $U_{last}$ is a vector $S$-group scheme, is central in $U$
 and is $\Aut(G,P)$--equivariant.
\end{slemma}

\begin{proof} Without loss of generality, we can assume that $G$ and $P$ are of constant type. 
Since the required properties are local for the \'etale topology 
on $S$, it is convenient to reason by a descent argument.

 We assume first that $G= G_1 \times_S  G_2 \dots \times_S G_c$ where $G_i$ is an adjoint semisimple
 $S$--scheme  whose absolute Cartan-Killing type is  connected.
 We have $P={P_1 \times_S P_2 \cdots \times_S P_c}$ and 
 put $U_{last}=  U_{1,last} \times_S \dots \times_S U_{c,last}$.
 This is a vector $S$--group scheme, central in $U=\rad_u(P)$ and 
 we claim that it is $\Aut(G,P)$-equivariant.
 
 Let $\phi \in \Aut(G,P)(S)$. Up to localization, we may assume that there exist a permutation 
 $\sigma$ in $c$ letters and
 $S$--isomorphisms $\phi_i: G_{\sigma{(i)}} \simlgr G_i$ ($i=1,\dots,c$)  such that 
 $\phi( g_1,\dots, g_c)= \bigl( \phi_1( g_{\sigma(1)}), \dots  , \phi_r( g_{\sigma(r)}) \bigr)$.
 Since each $U_{i,last}$ is characteristic in $P_i$, it follows that 
 $\phi( U_{last})=U_{last}$.
 This shows that    
$U_{last}$ satisfies the requirement. This is 
a vector $S$-group scheme  which is central in $U=\rad_u(P)$ by construction.

\smallskip

\noindent{\it General case.} Locally for the \'etale topology 
over $S$, $(G,P)$ is isomorphic to $(G_0, P_0)$ where $G_0$ is an adjoint Chevalley $S$-group
scheme  and $P_0$ is a standard parabolic subgroup 
of $G_0$ \cite[XXII.2.3 and XXVI.3.3]{SGA3}.
We denote by $U_0$ its unipotent radical and  by $U_{0, last} \subset U_0$
the $S$-subgroup defined in  the first case. 
By faithfully flat descent $U_{0, last}$ descends 
to $U_{last} \subset P$ and  satisfies the required property. 
\end{proof}

\subsection{Subgroups attached to parabolic subgroups}\label{subsec_attached} 

Let $R$ be a ring and
let $G$ be a reductive $R$--group scheme.
Let $P$ be an $R$--parabolic subgroup of $G$; $P$ admits a  Levi subgroup $L$
\cite[XXVI.2.3]{SGA3}. We consider   
the corresponding  opposite $R$--parabolic  subgroup $P^-$  to $P$.
We denote by $E_P(R)$ the  subgroup of $G(R)$ which is generated
by $\rad_u(P)(R)$ and $\rad_u(P^{-})(R)$; it does not depend of the choice of $L$
since Levi subgroups of $P$ are  $\rad_u(P)(R)$-conjugated.

\begin{sremark} {\rm  If $R$ is a field, and $P$ is a strictly proper
parabolic $R$-subgroup, then $E_P(R)$  does not depend on  the choice of $P$.
In this case,  the group $E_P(R)=G^{+,P}(R)$ is denoted by $G^{+}(R)$ \cite[prop.\ 6.2]{BT}.
}
\end{sremark}

\subsection{Generation of the Lie algebra: the field case}
Let $F$ be a field.
If $F$ is finite, 
we use the notation $F_{r}$ for the unique extension of $F$ of degree $r$.

\begin{slemma}\label{lem_root2} Let $G$ be a simply connected semisimple algebraic $F$-group
with Lie algebra $\gg$.
Let $P$ be a strictly proper parabolic $F$--subgroup. Let  
$U_{last}$ be the $F$--subgroup of $\rad_u(P)$ constructed in Lemma \ref{lem_last}.

\smallskip

\noindent (1) If $G$ is split, then  $G(F) \cdot \Lie(U_{last})$ generates 
the $F$-vector space $\gg$.

\smallskip

\noindent (2)  $\gg$ is the unique $G$--submodule  of 
 $\gg$ containing $\Lie(U_{last})$.

\smallskip

\noindent (3) If $F$ is infinite, then 
$E_P(F). \Lie(U_{last})$ generates 
the $F$-vector space $\gg$.

\smallskip

\noindent (4) If $F$ is finite, we have $E_P(F_r)=G(F_r)$ for each  $r \geq 1$ and 
the quantity 
$$
r_F(G,P)= \Inf\Bigl\{ r \geq 1 \, \mid \, \hbox{\enskip the $F_{r}$-vector space \enskip} 
\gg \otimes_{F} F_{r}
\hbox{\enskip is generated by $E_P(F_r) \cdot  \Lie(U_{last})(F_r)$} \Bigr\}
$$
 is $<\infty$.

\end{slemma}

\begin{proof}
(1) We can assume  that $G$ is almost simple.
Let $B$ be a Borel subgroup of $P$ and let $T$
be a maximal $F$--torus of $B$. 
Let $\alpha$ be the maximal root of $\Phi(G,T)$, we have $U_{\alpha} \subset U_{last}$ by construction.
Since $\alpha$ is a long root, 
the Lie subalgebra $\gg_\alpha=\Lie(U_\alpha)$ is called a long root subalgebra.
 Since roots of maximal length are conjugated under the Weyl group
and since maximal split tori of $G$ are $G(F)$-conjugated, 
it follows that all long root subalgebras  are 
$G(F)$--conjugated. According to \cite[lemma 1.1]{V} (based on \cite{Hi}), if $G$ 
is not of rank one,
the long root subalgebras generate the $F$--vector space $\gg$.
Thus  $G(F) \cdot \Lie(U_{last})$ generates 
the $F$-vector space $\gg$. 

It remains then to deal with the 
case of $\SL_2$ and its standard Borel subgroup  $P=B$.
We observe that   $\begin{pmatrix}
0 & 1\\
0 & 0
\end{pmatrix}$,
$\begin{pmatrix}
0 & 0\\
1& 0
\end{pmatrix}$ and $\begin{pmatrix}
-1 & 1\\
-1& 1
\end{pmatrix}$ are long root elements and form a $F$--basis of $\gsl_2$.

\smallskip

\noindent (2) We can assume that $F$ is algebraically closed so that the statement readily  follows from (1).

 \smallskip
 
 \noindent (3) This follows from the fact that $E_P(F)$ is Zariski dense in $G$ \cite[cor. 6.9]{BT}.

 \smallskip
 
 \noindent (4) Since $F$ is finite, $G$ is quasi-split and we have $E_P(F_r)=G(F_r)$ for each $r \geq 1$ according to \cite[1.1.2]{T}.
  Then (1) shows that $G( F_{\infty})\cdot \Lie(U_{last})(F_{\infty})$ generates 
 the $F_{\infty}$-vector space $\gg \otimes_{F} F_{\infty}$.
 Then there exists $r \geq 1$ such that  $G( F_{r})\cdot \Lie(U_{last})(F_r)$ generates 
 the  $F_{\infty}$-vector space  $\gg \otimes_{F} F_{\infty}$
 and a fortiori    the  $F_{r}$-vector space  $\gg \otimes_{F} F_{r}$.
\end{proof}

\begin{sremark} \label{rem_root3} {\rm Under the  hypothesis of Lemma \ref{lem_root2}, we assume furthermore that $F$ is finite field.

\smallskip

\noindent (a)  We have $r_{F_u}(G,P) = \Sup( 1,  r_{F}(G,P) - u)$ for each $u \geq 1$.

\smallskip

\noindent (b)  If  $G$ is split, Lemma \ref{lem_root2}.(1) is rephrased by the formula
$r_F(G,P)=1$.

 }
\end{sremark}

\subsection{Generation of the Lie algebra: Case of a semilocal ring}

The next statement is a  variation
of a result of Borel-Tits on the Whitehead groups over local fields 
\cite[prop.\ 6.14]{BT}.

\newpage

\begin{slemma}\label{lem_prod} Let $R$ be a semilocal ring. Let $G$ be a semisimple  $R$--group scheme equipped with 
a strictly proper $R$--parabolic subgroup $P$ of $G$.
We assume that the fundamental group of $G$ is \'etale.
Let $U_{last}$ be the $R$--subgroup $P$ defined in Lemma \ref{lem_last}.
Let $s_1, \dots, s_t$ be the closed points of $\Spec(R)$.
For each $j$ such that 
$\kappa(s_j)$ is finite 
we assume that $G_{\kappa(s_j)}$  that $r(G_{\kappa(s_j)}, P_{\kappa(s_j)})=1$.

\smallskip

\noindent (1) There exist $g_1, \dots, g_m \in E_P(R)$ such that the 
product map
$$
h: (U_{last})^m \to G, \enskip (u_1,\dots, u_m) \, \mapsto  \,  {^{g_1}\!u_1} \dots  
{^{g_m}\!u_m}
$$
is smooth  at $(1,\dots,1)_{s_j}$  for $j=1,\dots,n$.

\smallskip

\noindent (2) 
%If $R$ is semilocal,
The map $dh: \Lie(U_{last})(R)^m \to \Lie(G)(R)$ 
is onto.

\end{slemma}

\begin{proof} 
We denote by $\gm_1, \dots, \gm_t$ the maximal ideals of $R$,  by 
$\kappa_j=\kappa(s_j)=R/\gm_j$ for $j=1,\dots,t$.

\smallskip

\noindent (1) The hypothesis on the fundamental group of $G$
implies that $G^{sc} \to G$ is \'etale 
and then  reduces to the simply connected case.
Lemma \ref{lem_root2}.(1) implies that 
${E_P(\kappa_j) \cdot \Lie(U_{last})(\kappa_j)}$ generates the 
$\kappa_j$--vector space  $\Lie(G)(\kappa_j)$ if $\kappa_j$ is infinite
and this holds as well in the finite case granting to our assumption
$r( G_{\kappa_j}, P_{\kappa_j})=1$.

\begin{sclaim}\label{claim_KT} The map $E_P(R) \to \prod_j E_P(\kappa(s_j))$ is onto. 
\end{sclaim}

Let $P^{-}$ be an opposite parabolic subgroup scheme of $P$ and let 
$U^-$ be its unipotent radical.
Since $E_P(R)$ (resp.\  each $E_P(\kappa(s_j))$)
is generated (by definition) by $U(R)$ and  $U^{-}(R)$,
it is enough to show the surjectivity of 
$U(R) \to \prod_j U(\kappa(s_j))$. 
According to \cite[XXVI.2.5]{SGA3}, there exists
a finitely generated  locally free $R$--module $\cE$ such that 
$U$ is isomorphic to $\WW(\cE)$ as  $R$-scheme.
Since $\WW(\cE)(R)= \cE$ maps onto $\prod_j \WW(\cE)(\kappa(s_j))
=  \prod_j \cE  \otimes_R \kappa(s_j)$, the Claim is established.

There are $g_1, \dots, g_m \in  E_P(R)$
such that $\Lie(G)(\kappa(s_j))$ 
is generated by the $^{g_{i}}\!\Lie(U_{last})(\kappa(s_j))$ for $j=1,\dots,t$.
The differential of the product map
$$
h: (U_{last})^m \to G, \enskip (u_1,\dots, u_m) \, \mapsto \, {^{g_1}\!u_1} \dots  
{^{g_m}\!u_m}
$$
is
$\Lie(U_{last})^m \to \Lie(G), \enskip (x_1,\dots, x_m) \mapsto {^{g_1}\! x_1} \dots  
{^{g_m}\! x_m}$.  It is  onto by construction; we conclude that  $h$ is smooth at 
$(1,\dots,1)_{s_j}$ for $j=1,\dots,t$. 

\smallskip

\noindent (2) 
Then $J={\gm_1\cap \dots \cap \gm_t}$ is the Jacobson radical 
of $R$ and $R/J \cong  R/\gm_1  \times \dots R/\gm_t$.
Statement (1) shows that 
the map  $dh: \Lie(U_{last})(R)^m \to \Lie(G)(R)$ is surjective
modulo $\gm_i$ for $i=1,\dots,t$ so is surjective modulo $J$.
Since $\Lie(G)(R)$ is finitely generated, Nakayama's lemma
\cite[II.4.2.3]{Ks}  enables us to conclude that   $dh$ is onto.

\end{proof}

\section{Moduli stack of $G$--torsors}\label{section_moduli}

\subsection{Setting} Let $S$ be a noetherian separated base scheme.

Let $f: X \to S$ be a  proper  flat (finitely presented)
scheme   satisfying the resolution property fppf locally over $S$, i.e. every quasi-coherent $\cO_X$-module 
of finite type is the quotient of a finite locally free $\cO_X$-module fppf locally over $S$ \cite[Tag 0F86]{St}. 
This property is satisfied if $X$ is projective over $S$ 
and  also if $X$ is a divisorial scheme  \cite[II.2.24]{SGA6}.
This applies in particular to the case $X$ noetherian regular ({\it ibid}, II.2.7.1.1).

% Since $S$ is noetherian, $X \to S$ is of finite presentation \cite[Tag 01T7]{St}.

% We relax the property by requiring to hold locally for fppf.

\begin{slemma}\label{lem_nitsure}  Let $Y \to X$ be an affine scheme  of finite type.
 Then the $S$--functor $\prod_{X/S}Y$ is representable by an affine $S$-scheme
 of finite presentation.
\end{slemma}

\begin{proof} In the projective case, this is \cite[\S 1.4]{Hd2}. 
According to the faithfully flat descent theorem, this assumption is local 
for the flat topology. This permits to assume that 
$f: X \to S$ satisfies the resolution property.
This matters  in the following basic statement \cite[7.7.8, II.7.9]{EGA3},
see also \cite[\S 5.3, Th. 5.8]{FGAE} and its proof: if $\cF$ is  flat coherent sheaf over $X$, 
the fppf $S$-sheaf $T \mapsto H^0( X \times_S T, \cF_{X \times_S T})$ is representable
by a linear $S$-scheme.

According to \cite[1.7.15]{EGA2}, there  exists a closed immersion
$Y \to V( \cF)$ where $\cF$ is a quasi-coherent $\cO_X$--module of finite type.
The resolution property allows us to assume
that $\cF$ is a quotient of a finite  locally free $\cO_S$-module $\cE$.
Since $V( \cF) \to V(\cE)$ is a closed immersion \cite[9.4.11.(v)]{EGA1}, we get a closed embedding
$Y \to V(\cE)$. 
For each $S$-scheme $T$, we have 
$\Bigl( \prod_{X/S}V(\cE) \Bigr)(T)= V(\cE)( X\times_S T)= { H^0\bigl( X \times_S T, \cE^\vee_{X \times_S T} \bigr)}$.
Since $\cE^\vee$ is a coherent flat module over $X$, the 
$S$-functor $\prod_{X/S}V(\cE)$ is representable by an affine 
$S$--scheme. According to \cite[\S 7.6, prop. 2.(2)]{BLR}, $\prod_{X/S}Y$ is representable by a closed
$S$--subscheme of $\prod_{X/S}V(\cE)$.

The $S$--scheme $\prod_{X/S}Y$ is locally of finite presentation since its functor of points commutes 
with colimits;  the affine $S$--scheme $\prod_{X/S}Y$ is then of finite presentation.
\end{proof}

 Let $G$ be a  smooth affine group scheme over $X$.
%Let $p: C \to S$ be a  projective flat
%relative curve whose  geometric fibers are algebraic curves.
%We assume that $C$ is integral and that
%the map $\cO_S \to p_*(\cO_C)$ is universally an  isomorphism.
Let $BG=[X/G]$ the algebraic $X$-stack of $G$--bundles on $X$ \cite[8.1.14]{O};
it is smooth \cite[Tag 0DLS]{St}.
For each $X$-scheme $Y$, $B_G(Y)$ is the groupoid of $G_Y$-torsors.
We denote by $\Bun_G= \prod_{X/S}B_G$ the 
$S$--stack of $G$-bundles, i.e. $\Bun_G(T)=  B_G( X \times_S T)$ for each
$S$--scheme $T$.

\begin{slemma}\label{lem_weil}
 Let $P$, $Q$ be two $G$-torsors over $X$. Then 
the fppf $S$--sheaf 
$$
T \mapsto \Isom_{G_{X \times_S T}}\bigl( P_{X \times_S T}, Q_{X \times_S T} \bigr)
$$
is representable by an $S$-scheme  $\Isom^\flat(P,Q)$ 
which is  affine of finite presentation.
\end{slemma}

\begin{proof}  
Since the $G$--torsors $P$ and $Q$ are locally isomorphic over $X$ to $G_{X}$ with respect to the \'etale topology,
the faithfully flat descent theorem shows that the 
$X$--functor
$$
D \mapsto 
\Isom_{G\times_C D}\bigl( P \times_{X} D, Q \times_{X} D \bigr) 
$$
is representable by an affine smooth $X$--scheme.
We denote it by $\mathbf{Isom}_G(P,Q)$.
The $S$--functor
$$
T \mapsto \Isom_{G_{X_{T}}}\bigl( P_{X_{T}}, Q_{X_{T}} \bigr)
$$
is  nothing but the scalar restriction 
$\prod\limits_{X/S}\mathbf{Isom}_G(P,Q)$.
According to Lemma \ref{lem_nitsure}, this $S$-functor is representable by a
an affine $S$-scheme of finite presentation.
\end{proof}

\begin{sproposition}\label{prop_stack}  The $S$--stack $\Bun_G$ is 
 algebraic  locally of finite type with affine diagonal.
\end{sproposition}

\begin{proof}
(1) This is an application of the general result by Hall-Rydh on Weil restriction of algebraic stacks
\cite[Th. 1.2]{HR}. More precisely we deal with   $BG \xrightarrow{g} X \xrightarrow{f}$ and
we claim that 

\smallskip

(*) $f \circ g$ is locally of finite presentation and has affine diagonal.

\smallskip

Since $BG$ is smooth over $X$, $g$ is locally of finite presentation.
Locally of finite presentation is stable under composition so $f \circ g$ is locally of finite presentation.
The affineness of the  diagonal follows from  Lemma \ref{lem_weil}.
In particular, $f \circ g$ has affine stabilizers (i.e. the diagonal of $BG \to S$ 
has affine fibers). The quoted result shows that the $S$--stack $\Bun_G$ is 
 algebraic  locally of finite type and  with affine diagonal.
\end{proof}

We assume from now on that $X=C$ is a relative curve. 
According to \cite[Tag 0DMK]{St}, $C$ is locally $H$--projective (that is embeds 
in a projective space) for the \'etale topology so satisfies 
the resolution property \'etale locally (and a fortiori fppf locally) over $S$.

\begin{sproposition}\label{prop_stack2} The $S$--stack $\Bun_G$ is a smooth 
 algebraic stack locally of finite type with affine diagonal.
\end{sproposition}

\begin{proof} The smoothness remains to be established. We  use the criterion of formal smoothness
\cite[2.6]{He2} (or \cite[Tag 0DNV]{St}). We are given an $S$--algebra $R$ 
which is local Artinian with maximal ideal
$\gm$ such that $\gm^2=0$ and a  $G$-torsor $P_0$
over $C_0=C \times_S R/\gm$. We put $G_0= G_{C_0}$
and denote by $H_0=  P_0 \wedge^{G_0} G_0$ the
twisted group scheme over $C_0$ with respect 
to the action of $G_0$ on itself by inner automorphisms.
According to \cite[th.\ 8.5.9]{I}, 
the obstruction to lift $P_0$ in a 
$G$-torsor  over $C \times_S R$ is a class of
$H^2\bigl(C_0,  \cLie(G_0) \otimes_{\cO_{C_0}} \gm \bigr)$. 
But  $R/\gm=\kappa$ is a field and $C_0$ is of
dimension $1$ so that this group vanishes according 
to  Grothendieck's vanishing theorem
\cite[III.2.7]{Ha}.
The formal smoothness criterion is satisfied so that the 
algebraic stack $\Bun_{G}$ is smooth.
\end{proof}

\subsection{The tangent stack}
We consider now the tangent stack $T(\Bun_G/S)$ \cite[\S 17]{LMB}
which is algebraic ({\it loc. cit.}, 17.16).
By definition, for each $S$--scheme $Y$,
we have $T(\Bun_G/S)(Y)= \Bun_G( Y[\epsilon])$ where
$Y[\epsilon]= Y \times_\ZZ \ZZ[\epsilon]$. It comes with 
two $1$--morphisms 
$$
 \tau : T(\Bun_G/S) \to \Bun_G
$$
and $\sigma: \Bun_G \to T(\Bun_G/S)$.

\begin{sremark}{\rm
We can consider the smooth-\'etale site 
on $\Bun_G$ and the quasi--coherent sheaf
$\Omega^1_{\Bun_G/S}$;  its associated generalized vector
bundle $\bV( \Omega^1_{\Bun_G/S})$  is an algebraic  $S$-stack.
There is a canonical $1$-isomorphism between
 $T(\Bun_G/S)$ and $\bV( \Omega^1_{\Bun_G/S})$ ({\it loc. cit.}, 17.15).
 We shall not use that fact in the paper.
 }
\end{sremark}

Our goal is the understanding of the fiber product of $S$--stacks
$$
  T_b\Bun_G= T(\Bun_G/S)  \times_{ \Bun_G} S
$$
where $b: S \to  \Bun_G$ corresponds to the trivial $G$--bundle $G$ over $C$. 
According to \cite[2.2.2]{LMB}, for each $S$--algebra $B$, 
the fiber  category $T_b\Bun_G(B)$  has for objects the couples
$(F,f)$ where $F$ is a $G_{C_{B}}[\epsilon]$--torsor
and $f: G_{C_{B}} \simlgr F \times_{C_{B[\epsilon]} }C_{B}$ 
is a trivialization of
$G_{C_{B}}$--torsors; an arrow $(F_1, f_1) \to (F_2, f_2)$
is a couple $(H, h)$ where \break $H: F_1 \simlgr F_2$ is an isomorphism
of $G_{C_{B[\epsilon]}}$--torsors and 
$h \in G(C_{B})$ with the compatibility 
$\bigl( H \times_{C_{B[\epsilon]}} C_{B} \bigr) 
\circ f_1= f_2 \circ h$.

\subsection{Relation with the Lie algebra}
\label{relation}

We consider the Weil restriction  $G'= \prod_{C[\epsilon]/C} G_{C[\epsilon]}$, this 
is an affine smooth  $C$--group scheme \cite[A.5.2]{CGP}. It comes with a $C$--group homomorphism
$j: G  \to G'$ and with a  
$C[\epsilon]$-group homomorphism $q: G' \times_C C[\epsilon] \to 
G \times_S S[\epsilon]$.

We consider the functor $\Phi$ between the categories of $G'$-torsors over $C$ 
and that of $G$--torsors over $C[\epsilon]$ defined
by the assignment defined
by the assignment $E' \mapsto 
q_*\bigl( E' \times_C C[\epsilon] \bigr)$.
According  to \cite[VII.1.3.2]{Gd}, the restriction 
map $H^1( V[\epsilon], G_{V[\epsilon]}) \to H^1(V,G)$ is 
bijective for each affine $S$-scheme $V$;
it follows that each $G$--torsor over $C[\epsilon]$
is trivialized by an \'etale cover of $C$ extended to  $C[\epsilon]$.
According to \cite[XXIV.8.2]{SGA3} (see also \cite[III.3.1.1]{Gd}), it follows that 
we can define the functor $\Psi$ 
by the assignment $F' \mapsto  \prod\limits_{C[\epsilon]/C}( F')$.
The functors $\Phi$ and $\Psi$ are inverse of each others so that 
the groupoids  $\Tors_{G'}(C)$ and $\Tors_{G \times_C C[\epsilon]}(C[\epsilon])$
are isomorphic.

We come now to Lie algebras considerations.
By definition of the Lie algebra, the $C$--group $G'$ fits
in a split exact sequence of $C$--group schemes
$$
0  \to \WW(\cLie(G)) \xrightarrow{\e} G' \xrightarrow{\pi} G \to 1
$$
where $\cLie(G)=\omega_{G/S}^\vee$ (see \S \ref{app_lie} and  Remark \ref{rem_lie2}.(a)).

According to \cite[III.3.2.1]{Gd} we have an equivalence of groupoids between
$\Tors_{\WW(\cLie(G))}(C)$ and
that of couples $(E', \eta)$ where $E'$ is  a 
$G'$--torsor over $C$ and 
$\eta: G \simlgr \pi_*{E'}$ is a trivialization.
Taking into account the previous isomorphism of categories,
we get  then an equivalence of groupoids between 
$\Tors_{\WW(\cLie(G))}(C)$ and
that of couples $(F, \xi)$ where $F$ is  a 
$G$--torsor over $C[\epsilon]$ and 
$\xi: G_C \to F \times_{C} C[\epsilon]$ is a trivialization;
the morphisms are clear.

We come back now to the previous section involving a 
$S$--algebra $R$ and the morphism $b:\Spec(R) \to \Bun_G$ associated
to the trivial $G$--torsor. By comparison it follows that 
the fiber category $T_b\Bun_G(R)$ is equivalent
to $\Tors_{ \WW(\cLie(G))(C)}$.

\section{Uniformization and local triviality}\label{section_unif}

This section presents in a slightly more general manner than
classical material on uniformization of $G$--bundles \cite{BL1, He1, He2, So}.

\subsection{Loop groups}
We continue with the framework of the previous section and  assume from now 
on  that $S=\Spec(R)$ is affine noetherian.
 We deal with   a (proper flat) 
relative curve $p: C \to \Spec(R)$; 
it satisfies  \'etale locally
the resolution property since it is locally $H$-projective \cite[Tag 0E6F]{St}. 

\medskip 

Let  $D$ be a  finite flat $S$--scheme  with
a closed embedding ${s: D \to C}$ such that 

\smallskip

(i) the complement $C \setminus D$ is affine and 
will be denoted by $V^D$; 

%$D$ is an effective Cartier divisor which is ample;

\smallskip

(ii) $s$ factorizes through an affine $R$--subcheme $V$ of 
$C$. 

\smallskip 

Note that (i) is satisfied if $D$ is an effective Cartier divisor which is ample.
Let $V=\Spec(A)$, $V^D=\Spec(A_D)$ and
$V \cap V^D  =\Spec(A_\sharp)$;
this intersection is affine because the morphism $C \to S$
is  separated \cite[Tag 01KP]{St}. 
We denote by $I \subset A$ the ideal defining $D$.
We consider the completed  ring $\widehat A:= \widehat A_{I}= \limproj_{n} A/I^n$.
We need some basic facts from commutative algebra
 (see \cite[III.4.3,  th.\ 3 and prop.\ 8]{BAC} for (a) and (b)).

\smallskip

\noindent (a) $\widehat A$ is noetherian and flat over $A$.

\smallskip

\noindent (b) The assignment $\gm \mapsto \gm \widehat A$ 
provides a  correspondence between the maximal 
ideals of $A$ containing $I$ and the maximal ideals of $\widehat A$;

\smallskip

\noindent (c) If $R$ is semilocal so is $\widehat A$.

\smallskip

\noindent If  $R$ is local, the finite $R$--algebra
  $A/I$ is semilocal so we get (c) from (b).

\medskip
    
\medskip

  We recall that 
 $G$ is  a  smooth affine group scheme over $C$.
We consider the following $R$--functors defined for each 
$R$--algebra $B$ by:

\medskip

(1) $L^+G(B)=G\Bigl( \widehat{(A \otimes_R B)}_{I  \otimes_R B} \Bigr)$;

\smallskip

(2) $LG(B)=G\Bigl( 
\widehat{(A \otimes_R B)}_{I  \otimes B} \otimes_{A} A_\sharp  \Bigr)$.

\begin{sexample}\label{ex_simple} {\rm  
(a) The simplest example of our situation is 
$C=\PP^1_R = V_{\infty} \cup V = \Spec(R[t]) \cup \Spec(R[t^{-1}])$
and for $D$ the point $0$ of $C$.
In this case, we have $A=R[t]$, $I=t A$, $A_\sharp= R[t, t^{-1}]$ and  
$\widehat A= R[[t]]$. For each $R$--algebra $B$, we have
$\widehat{(A \otimes_R B)}_{I  \otimes B}=B[[t]]$
and  $\widehat{(A \otimes_R B)}_{I  \otimes_R B} \otimes_{A} A_\sharp=
B[[t]] \otimes_{R[t] } R[t, t^{-1}]= B[[t]][\frac{1}{t}]$.
The standard notation for the last ring is $B((t))$.
}
\end{sexample}

\subsection{Patching}
For simplicity we assume that $S=\Spec(R)$ where $R$ is a noetherian ring.
If we are given  an $R$--algebra $B$ (not necessarily noetherian),
we need to deal with the rings
$\widehat{(A \otimes_R B)}_{I  \otimes B}$ 
and $\widehat{(A \otimes_R B)}_{I  \otimes B} \otimes_{A} A_\sharp$.
As pointed out by Bhatt \cite[\S 1.3]{Bh},
the Beauville-Laszlo theorem \cite{BL} states that
one can patch compatible quasi-coherent sheaves on 
$\Spec(\widehat{(A \otimes_R B)}_{I  \otimes B})$ and 
$V^D\times_R B$ to a quasi-coherent sheaf on $C_B$, 
provided the sheaves being patched are flat along $\Spec(B/IB)$. 
In particular the square of functors

\[\xymatrix@1{
\cC(C_B)\ar[r] \ar[d] & \enskip \cC\Bigl(\widehat{(A \otimes_R B)}_{I  \otimes B} \Bigr) \ar[d] \\
  \cC(V^D \times_R B)  \ar[r]& \enskip 
  \cC\Bigl( \widehat{(A \otimes_R B)}_{I  \otimes B} \otimes_A A_\sharp \Bigr) 
}\]
is  cartesian where $\cC(X)$ stands for the category
of flat quasi-coherent sheaves over the scheme $X$ 
(resp.\ the category of flat affine schemes over $X$).
Note that if the ring $B$ is noetherian, Ferrand-Raynaud's patching 
 \cite{FR} (see also \cite{MB}) does the job.

\begin{sproposition}\label{prop_patch}
(1) The square of functors

\[\xymatrix@1{
\Tors_G(C_B)\ar[r] \ar[d] & \enskip \Tors_G\Bigl(\widehat{(A \otimes_R B)}_{I  \otimes B} \Bigr) \ar[d] \\
 \Tors_G(V^D \times_R B)  \ar[r]& \enskip 
 \Tors_G\Bigl( \widehat{(A \otimes_R B)}_{I  \otimes B} \otimes_A A_\sharp \Bigr) 
}\]
is  cartesian.

\smallskip

\noindent (2)  The $S$--functor $LG$ is isomorphic to
the functor associating to  each  $R$--algebra  $B$  the $G$--torsors over 
 $C_B$ together with trivializations on $V^D \times_R B$ and on 
 $\Spec(\widehat{(A \otimes_R B)}_{I  \otimes B})$.
\end{sproposition}

\begin{proof}(1) Since $G$ is affine and flat over $C$,  it is a formal 
corollary of the patching statement.

\smallskip

\noindent (2) Let $\cC(B)$ be the the category of  $G$--torsors over 
 $C_B$ together with trivializations on $V^D \times_R B$ and on 
 $\Spec(\widehat{(A \otimes_R B)}_{I  \otimes B})$.
 An object of $\cC(B)$  is a triple $(E,f_1, f_2)$ where
 $E$ is a $G_{C_B}$--torsor, $f_1: G_{V^D \times_R B} \simlgr E_{V^D \times_R B}$ 
 and $f_2: G_{\widehat{(A \otimes_R B)}_{I  \otimes B}} \simlgr 
 E_{\widehat{(A \otimes_R B)}_{I  \otimes B}}$ are trivializations.
An element $g\in LG(B)=
 G\bigl( \widehat{(A \otimes_R B)}_{I  \otimes B} \otimes_A A_\sharp  \bigr)$
 gives rise to the right translation 
 $$
 (G_{V^D \times_R B})_{ \widehat{(A \otimes_R B)}_{I  \otimes B} \otimes_A A_\sharp}
 \simlgr 
 (G_{ \widehat{(A \otimes_R B)}_{I  \otimes B}})_{ \widehat{(A \otimes_R B)}_{I  \otimes B} \otimes_A A_\sharp}
 .
 $$
 It defines a $G_C$--torsor $E_g$ with trivializations $f_1$ and $f_2$
 on $V^D \times_R B$ and on 
 $\Spec(\widehat{(A \otimes_R B)}_{I  \otimes B})$.
 We get then a morphism  $\Phi: LG(B) \to \cC(B)$. 
 
Conversely let $c=(T,f_1,f_2)$  be an object of $\cC(B)$.
 Then the map $f_1^{-1}f_2: G_{ \widehat{(A \otimes_R B)}_{I  \otimes B} \otimes_A A_\sharp}
 \to G_{ \widehat{(A \otimes_R B)}_{I  \otimes B} \otimes_A A_\sharp}$
 is an isomorphism of $G$--torsors hence is the right translation by 
 an element $g=\Psi(c) \in LG(B)$.
 The functors $\Phi$ and $\Psi$ provide the desired  
 isomorphism of functors.
  % equivalence of categories.
\end{proof}

Continuing with the $R$-algebra  $B$, we have a factorization

\[\xymatrix@1{
 LG(B) \ar[rr]^p \ar[d] &&  \Bun_G(B) \ar[d]_{class \enskip map}\\
c_G(B):= G(V^D \times_R B) \,  \backslash \, LG(B) \, / \, 
L^+G(B) \qquad   \ar[rr]^{\qquad\qquad \underline{p}} && H^1(C_B,G). 
}\]
 The map $p$ is called the uniformization map. 
Proposition \ref{prop_patch}.(2) implies that  the bottom map induces 
  a bijection
$$
\leqno{(*)} \qquad \qquad
c_G(B) \simlgr \ker\Bigl( H^1(C_B,G) \to H^1(V^D \times_R B,G) \times
H^1\bigl( \widehat{(A\otimes_R B)}_{ I \otimes_R B},G \bigr) \Bigr). 
$$

\subsection{Link with the tangent space}\label{subsec_tg}
Our goal is to differentiate  the mapping $p: LG \to \Bun_G$.
Let $B$ be an  $R$--algebra and  consider the map  
$$
p[\epsilon]: LG(B[\epsilon])  \to \Bun_G(B[\epsilon]).
$$
We have $\widehat{(A \otimes_R B[\epsilon])}_{I  \otimes_R B[\epsilon]}
= \Bigl(\widehat{(A \otimes_R B)}_{I  \otimes B} \Bigr)[\epsilon]$
so that $LG(B[\epsilon])= 
{G\Bigl(\bigl(\widehat{(A \otimes_R B)}_{I  \otimes B} 
\otimes_A A_\sharp \bigr)[\epsilon]\Bigr)}$.
We consider the commutative diagram of categories\footnote{ 
 With the convention
that a set   defines a groupoid  
\cite[Tag 001A]{St}.
}

\[\xymatrix@1{
0 \ar[r] &   \Lie(G)\Bigl(\bigl(\widehat{(A \otimes_R B)}_{I  \otimes B} 
\otimes_A A_\sharp \bigr) \Bigr)
\ar[r]^{e}  & \ar[d]^{\wr} \ar[r]
G\Bigl(\bigl(\widehat{(A \otimes_R B)}_{I  \otimes B} 
\otimes_A A_\sharp \bigr)[\epsilon]\Bigr)
& \ar[d]^{\wr} G\Bigl(\bigl(\widehat{(A \otimes_R B)}_{I  \otimes B} 
\otimes_A A_\sharp \bigr) \Bigr) \\
 && LG(B[\epsilon]) \ar[r] \ar[d] & LG(B) \ar[d] \\
 && \Bun_G(B[\epsilon]) \ar[r] & \Bun_G(B)
}\]
where the first line is the exact sequence defining  
 the Lie algebra. By considering the fiber
 at the trivial $G$--torsor $b \in \Bun_G(B)$, we get then a functor 
$$
p\circ \e : \Lie(G)\Bigl( \widehat{(A \otimes_R B)}_{I  \otimes B} \otimes_A A_\sharp\Bigr) 
\to T_b\Bun_G(B). 
$$
Since $\Lie(G)\Bigl( \widehat{(A \otimes_R B)}_{I  \otimes B} 
\otimes_A A_\sharp\Bigr)=
\WW(\cLie(G))\Bigl( \widehat{(A \otimes_R B)}_{I  \otimes B} \otimes_A A_\sharp \Bigr) $, 
 we have an $R$--functor 
$$
dp: L\WW(\cLie(G))  \to T_b\Bun_G.
$$
We use now the equivalence of categories between  $T_b\Bun_G(B)$
and $\Tors_{\WW(\cLie(G))}(C_B)$ (cf.  \ref{relation}) and get 
the following compatibility with the classifying maps

\[\xymatrix@1{
L{\WW(\cLie(G))}(B) \ar[r] \ar[d]^{dp}& 
\WW(\cLie(G))(V^D \times_R B) \, \backslash \, 
L{\WW(\cLie(G))}(B) \, / \, 
L^+{\WW(\cLie(G))}(B)  \ar[d] &   \\ 
 T_b\Bun_G(B) \ar[r]^{class \, \,  map} & H^1(C_B, \WW(\cLie(G)).
}\]
We observe that the $\WW(\cLie(G))$--torsors over affine schemes
are trivial so that the top right  map is an isomorphism according to 
the fact $(*)$ above.
Also $H^1(C_B, \WW(\cLie(G)))$
identifies with the coherent cohomology of the 
$\cO_S$--module $\cLie(G)$ \cite[prop.\ III.3.7]{Mn}.

\subsection{Heinloth's section}

This  statement is a variation over a local henselian noetherian base
of a result due to Heinloth when the residue field is
 algebraically closed  \cite[cor. 8]{He1}.

\begin{sproposition} \label{prop_section} Assume that $S=\Spec(R)$ 
with $R$ local noetherian henselian  with
residue field $\kappa$.
We assume that $G$ is semisimple and that its  fundamental group 
is smooth over $C$.
We assume that $G_{D_\kappa}$ admits a 
strictly proper parabolic $D_\kappa$--subgroup $Q$
such that $n_{\kappa(s)}(G_{\kappa(s)}, Q_{\kappa(s)})=1$
for each point $s \in D_\kappa$ with finite residue field.

\smallskip

\noindent (1) There exists a map $F: \mathbf{A}^n_R \to LG$
such that  the composite
$$
f: \mathbf{A}^n_R \xrightarrow{F} LG \xrightarrow{p} \Bun_G
$$ 
is a map of stacks,   maps  $0_R$ to  the 
trivial $G$--torsor $b$  and such that 
$$
df_{0_R} : R^n \to T_b\Bun_G(R) 
$$
is essentially surjective. Furthermore there exists a neighborhood
$\cN$ of $0_\kappa$ in $\mathbf{A}^n_R$
such that $f_{\mid \cN}$ is smooth.

\smallskip

\noindent (2) Let $E$ be a $G$--bundle over $C$ such that 
$E \times_{C} C_\kappa$ is trivial. Then $E$ is trivial on $V^D$.

\end{sproposition}

\begin{proof} 
 (1) 
 The proof goes by a differential argument.
The $R$--module $H^1(C, \Lie(G))$ is finitely generated over
$R$ \cite[III.5.2]{Ha} and we lift a generating family of  $H^1(C, \Lie(G))$ to 
a family of elements $Y_1, \dots, Y_r$ of 
$\Lie(G)\Bigl(\widehat{A}\otimes_{A} A_{\sharp}  \Bigr)$.
We have noticed  that $\widehat A$ is a semilocal noetherian ring
(Ex. \ref{ex_simple}.(b)). We want now
  to apply Lemma \ref{lem_prod} to $G_{\widehat A}$
with respect to  the closed points of $\Spec(\widehat A)$. 
Let $Z$ be the $C$--scheme of strictly parabolic subgroups 
of $G$ \cite[XXVI.3]{SGA3}. It is a smooth proper $C$--scheme. It follows
that its Weil restriction $\prod_{D/R}(Z_D)$
its smooth so that 
Hensel's lemma shows that  $Q$ 
lifts to a strictly proper parabolic $D$--subgroup scheme $P_D$
 of $G_D$. Similarly $Z( \widehat A) \to Z(A/I)=Z(D)$ is onto
 %Once again Hensel's lemma  shows 
 so that $P_D$ lifts a strictly proper parabolic $\widehat A$--subgroup
scheme $P$ of $G_{\widehat A}$.
We put $U= \rad_u(P)$, it is a smooth affine $\widehat A$--group scheme
and we denote by $U_{last}$ its $R$--subgroup scheme constructed in \S \ref{subsec_attached}.

Lemma \ref{lem_prod} provides elements $g_1, \dots, g_m \in E_P(\widehat A)$
such that the  product map
$$
h: (U_{last})^m \to G, \enskip (u_1,\dots, u_m) \, \mapsto \, {^{g_1}\!u_1} \dots  
{^{g_m}\!u_m}
$$ 
induces a surjective differential
$$
dh: \Lie(U_{last})^m(\widehat A)  \, \to \,  \Lie(G)(\widehat A) , \quad (X_1,\dots,X_m) \, \mapsto  \,
{^{g_1}\!X_1} + \dots  + {^{g_m}\!X_m}.
$$
In other words we have 
$$
\Lie(G)(\widehat A) = \enskip  ^{g_1}\!\Lie(U_{last})(\widehat A)
\enskip  + \enskip  ^{g_2}\!\Lie(U_{last})(\widehat A)\enskip  +
\enskip \dots \enskip  +  \enskip  ^{g_m}\!\Lie(U_{last})(\widehat A)
$$
so that (using the identity of Lemma \ref{lem_lie2}.(2))
$$
\Lie(G)(\widehat A  \otimes_A A_\sharp)=
\Lie(G)(\widehat A)  \otimes_A A_\sharp 
\qquad  
$$
$$
= \enskip
^{g_1} \! \Lie(U_{last})(\widehat A) \otimes_A A_\sharp \enskip
+ \enskip ^{g_2} \! \Lie(U_{last})(\widehat A) \otimes_A A_\sharp
\enskip + \enskip \dots  \enskip + \enskip ^{g_m} \! \Lie(U_{last})(\widehat A) \otimes_A A_\sharp.
$$
We can write $$
\leqno{(**)} \qquad \qquad  
Y_i = \sum\limits_{j=1, \dots , m} c_{i,j} \enskip ^{g_j}\!Z_{i,j}
$$
where $Z_{i,j} \in \Lie(U_{last})(\widehat A)$ and
$c_{i,j} \in \widehat A  \otimes_A A_\sharp$ for each $j$.

Since $U_{last}$ is an $\widehat A$--vector group scheme, there is a canonical 
identification $\exp: \WW( \cLie(U_{last})) \simlgr  U_{last}, \enskip X \mapsto \exp(X)$. 
We consider the polynomial ring
$B= R[ t_{i,j}]$ where $i=1,\dots,r, j=1,\dots,m$.
We consider the map of $R$--functors 
$F:\mathbf{A}^{rm}_R \to LG$ defined 
by the element 
$$
\prod\limits_{i=1,\dots,r ; j=1,\dots,m}
{^{g_i}\!\exp\Bigl(t_{i,j} \enskip  c_{i,j} \enskip Z_{i,j} \Bigr)}
 \in 
G\Bigl(\widehat{(A \otimes_R B)}_{I \otimes_R B}
\otimes_{A} A_\sharp\Bigr)=LG(B)
$$
where we can take for example the lexicographic order.
It induces an $R$--map \break  $f: \mathbf{A}^{rm}_R \to \Bun_G$ of stacks
mapping $0_R$ to  the trivial $G$--bundle. 
Taking into account the last compatibility of \S \ref{subsec_tg} ,
its differential at $0_R$ 
$$
df: R^{rm} \to  T_b\Bun_G(R) 
$$
factorizes  through $L{\WW(\cLie(G))}(B)$. More precisely 
we have a commutative diagram

\[\xymatrix@1{
R^{rm} \ar[r]^{h \quad} \ar[d]^{df_0}  \qquad & \quad 
L{\WW(\cLie(G))}(R) \ar[r] \ar[ld]_{dp} & H^1(C, \Lie(G))  \\
 T_b\Bun_G(R)  \ar[rru]_{class \enskip map} &&
}\]
where  $h$ maps 
the basis  element $e_{i,j} \in R^{rm}$ to  $c_{i,j} \, ^{g_j} \! Z_{i,j}$ in 
$L{\WW(\cLie(G))}(R)$. We take into account the identity $(**)$.
By $R$--linearity, the image of $R^{mr} \to H^1(C, \Lie(G))$
contains all $Y_i$'s.
Since the $Y_i$'s generate the $R$-module 
$H^1(C, \Lie(G))$, we conclude that $df_0$ is essentially surjective.

The formation of $R^1p_*\cLie(G)$
commutes with base change, we have an 
isomorphism  $H^1(C, \cLie(G)) \otimes_R \kappa
\simlgr H^1(C_0, \cLie(G))$
so that 
$df_{0,k}: k^{mr} \to H^1(C_0, \Lie(G))$ is onto as well.

It follows that  
 $f$ is smooth locally at $0_\kappa$ according to the Jacobian smoothness criterion
 \ref{prop_jacobian} stated in the appendix.
Thus there is $\cN$ as claimed in the statement.

\smallskip

\noindent (2) We see  $E$ as an object of  $\Bun_G(R)$ and  
consider 
the fiber product
 \[\xymatrix@1{
\cN \ar[r]^{f}& \Bun_G  \\
 Y \ar[r]^{\pi \quad } \ar[u]& \Spec(R) \ar[u]^{E} .
}\]
Then $Y$ is an $R$--algebraic space \cite[8.16]{O}
which is smooth over $R$.
Let $Q$ be the $G$-torsor over $C_\cN$ defined by $f: \cN \to \Bun_G$. 
 The algebraic space $Y$ is representable
 by the $\cN$--scheme  $\Isom^\flat(Q,E_\cN )$ defined in Lemma \ref{lem_weil}.
   Hensel's Lemma  
 shows  that $Y(R) \not=\emptyset$. It follows 
 that there exists $u \in \cN(R)$ which maps to $E$.
 Since the map $\cN \to \Bun_G$ factorizes  through $LG$, 
 we conclude that the $G$-torsor $E$ is trivial on $V^D$. 
 \end{proof}

\section{Proof of the main result}\label{section_result}
 
\medskip

We need the following consequence of  Poonen's result \cite{P} 
and its refinement  by  Moret-Bailly  \cite{MB2}.

\begin{sproposition}\label{prop_poonen} 
Let $F$ be a field.
Let $Y$ be an irreducible  $F$-scheme of finite type of positive dimension. 
We denote by $Y_0^{sep}$ the set of separable closed points of $Y$.
Let $X$ be an $F$--scheme of finite type
and let  $f: X \to Y$ be a smooth surjective $F$-morphism.
Then
\[
\Sigma_f(Y) = \Bigl\{ y \in Y_0 \, \mid \, X_y(F(y)) \not= \emptyset 
 \Bigr\}
\]
is Zariski dense in $Y$.
If furthermore $Y$ is $F$-smooth, then  the  set 
\[
\Sigma^{sep}_f(Y) = \Bigl\{ y \in Y_0^{sep} \, \mid \, X_y(F(y)) \not= \emptyset 
 \Bigr\}
\]
is Zariski dense in $Y$.
\end{sproposition}

\begin{proof}
 Since shrinking is allowed, this reduces to show that the set 
 $\Sigma_f(Y)$ is non-empty which  is  \cite[remark p. 225]{P}. 
 The second fact  uses the refinement of \cite[page 472]{MB2}. 
\end{proof}

\begin{slemma}\label{lem_qs}
Let $F$ be field. Let  
$X$ be an irreducible  $F$-scheme of finite type and 
of dimension $\geq 1$.  Let $X_0^{sep}$ be the set of separable
closed points of $X$.
Let $H$ be a semisimple
$X$--group scheme. Then the set 
\[
X(s)= \Bigl\{ x \in X_0 \, \mid \, 
\hbox{$H_{F(x)}$ is split } \Bigr\} 
\]
is Zariski dense in $X$. 
If furthermore, $X$ is smooth, then the set 
\[
X^{sep}(s)= \Bigl\{ x \in X_0^{sep} \, \mid \, 
\hbox{$H_{F(x)}$ is split } \Bigr\} 
\]
is Zariski dense in $X$. 
\end{slemma}

 \begin{proof} 
 We  consider the $X$--scheme $Y=\Isom(H_0,H)$ 
 (where $H_0$ is  the Chevalley form of $H$) which is affine smooth over $X$
 \cite[XXIV.1.9]{SGA3}.
 Proposition \ref{prop_poonen} applied to $Y \to X$
 yields that $X(s)$ is dense in $X$.
 \end{proof}

We can proceed to the proof of Theorem \ref{thm_main}.
  
\begin{proof} 
We are given two $G$-torsors over $C$ 
such that 
$E \times_{C} C_\kappa$ is isomorphic
to $E' \times_C C_\kappa$. Up to consider the twisted $R$--group scheme ${^{E'}G}$, 
we can assume that $E'=G$ without loss of generality.
Let $\Theta$ be the set of irreducible
components of  $C_\kappa$  and denote by $C_\kappa^\theta$ the
component attached to $\theta \in \Theta$.

\smallskip

\noindent{\it Case (I).} 
Since each $C^{sm} \cap  C^{\theta}_\kappa$ is smooth
and nonempty, 
Lemma \ref{lem_qs} provides two fully distinct families of closed   
separable points $(c^\theta_1)_{\theta \in \Theta}$
and $(c^\theta_2)_{\theta \in \Theta}$ of 
$C^{sm} \cap  C_\kappa^\theta$ 
such that $G_{c^\theta_i}$ is a split semisimple
$\kappa(c^\theta_i)$--group for $i=1,2$ and each $\theta \in \Theta$. 
Let $Q^\theta_i$ be a $\kappa(c^\theta_i)$-Borel subgroup of $G_{c^\theta_i}$
for each  $\theta$ and  $i=1,2$.
Taking into account Lemma \ref{lem_root2}.(1), we have
$n_{\kappa(s)}(G_{\kappa(c^\theta_i)}, Q_{\kappa(c^\theta_i)})=1$
if $\kappa$ is finite (that is, there are enough long root elements).
%For each $\theta_i$, there exists a finite local $R$-extension $R^\theta_i$ 
%whose residue field is $\kappa(c^\theta_i)$ \cite[IX.37.2]{BAC}.
Since $R$ is henselian  there exists finite \'etale extensions $R^\theta_i$ of 
$R$ which lifts $\kappa(c^\theta_i)/\kappa$ and
$R^\theta_i$ is henselian as well 
\cite[Tag 04GH]{St}.
Since $C^{sm}$ is smooth over $R$,
Hensel's lemma applies to $C^{sm} \times_R R^\theta_i$
shows that each $c^\theta_i$ 
lifts in a closed $R$--subscheme
$D^\theta_i \to C$ which is finite \'etale over $R$.
We put $D_i = \bigsqcup\limits_{\theta \in \Theta} D^\theta_i$ for $i=1,2$.

 Since $C$ is projective over $R$ and $D_i$ is semilocal,
 $D_i$ is a closed $R$--subscheme of 
 an affine open $R$--subscheme of $C$ and $D_i$ is finite \'etale over $R$.
 For each point $s \in \Spec(R)$,  $D_{i,s}$ consists of smooth points of $C_s$
 so is an effective Cartier divisor, hence $D_i$ is a relative Cartier divisor for 
 $C/R$ \cite[end of Tag 062Y]{St}.
 By construction,  $D_{i,\kappa}$  has positive degree on each irreducible 
component of $C_\kappa$ so is ample \cite[\S 7.5, prop. 5.5]{Liu} so that 
$D_i$ is ample \cite[4.7.1]{EGA3}  and hence $C \setminus D_i$ is affine for $i = 1, 2$.
 
 We claim that the scheme  $C \setminus D_i$ is affine for $i=1,2$.
 %Since   $D_{i, \kappa}$ is an ample divisor
%of the curve $C_\kappa$,  $D_i$ is an ample divisor
% on $C$ \cite[4.7.1]{EGA3} 
 The group   $G\times D_{i, \kappa}$  admits a 
Borel subgroup (resp.\ is split) for $i=1,2$.
Now let $E$ be a $G$-torsor over $C$ such that
$E \times_C C_\kappa$ is trivial. 
Proposition \ref{prop_section}.(2) shows that $E_{\mid C \setminus D_i}$
is trivial for $i=1,2$. Since $C={(C \setminus D_1) \cup (C \setminus D_2)}$, 
we conclude that the $G$--torsor $E$ is locally trivial for the Zariski topology. 

\smallskip

\noindent {\it Case (II).} In this case $R$ is a henselian DVR.
Lemma \ref{lem_qs} provides
two fully distinct families of closed 
points $(c^\theta_1)_{\theta \in \Theta}$
and $(c^\theta_2)_{\theta \in \Theta}$ of 
$C_\kappa$ 
such that $G_{c^\theta_i}$ is a split semisimple
$\kappa(c^\theta_i)$--group for $i=1,2$ and for each $\theta \in \Theta$. 

We use now that there is closed $R$--embedding $i: C \to \PP^N_R$ \cite[Tag 0E6F]{St}.
We note that $C \to \Spec(R)$ is a regular fibered surface in the sense of Liu's book \cite{Liu}. 
According to \cite[\S 8.3, lemma 3.35]{Liu},
there exists an effective Cartier (equivalently Weil) ``horizontal''
divisor $D^\theta_i$ of $C$  such that
$C_\kappa \cap \mathrm{Supp}(D_i^\theta)= c_i^\theta $ for 
each $\theta \in \Theta$ and $i=1,2$  (note it is finite flat over $S$). 
We consider the  effective Cartier  divisors 
$D_i = \bigsqcup\limits_{\theta \in \Theta} D^\theta_i$ for $i=1,2$;
$D_i$ is finite flat over $S$.
According to \cite[Tag 056Q]{St}, $D_i$ is a relative Cartier divisor
on $C$ and $D_{i,\kappa}$ is an effective Cartier divisor of $C_\kappa$.
Since $G_{\kappa(c_i)}$ is split,
we have that $G\times D_{i, \kappa}$  is split for $i=1,2$ 
by using the smoothness
of the scheme $\Isom(G_0,G)$ (where $G_0$ is the Chevalley form of $G$).
Repeating verbatim the argument of Case (I) finishes the proof.
\end{proof}

\begin{sremark} {\rm 
In the proof of (1), 
an important step is the construction of 
the divisor $D$ such that $D$ is finite \'etale over $R$ 
and $G_D$ is split. Though our contruction is quite different, 
a similar argument has been used by Panin and Fedorov in 
their proof of Grothendieck-Serre's conjecture \cite[prop. 4.1]{FP}.
}
\end{sremark}

\section{Extension to reductive groups}\label{section_gather}
 
We gather here our results in a single long statement.

\begin{stheorem}\label{thm_red} Assume that
$R$ is local henselian noetherian of  residue field $\kappa$.
Let $p$ be the characteristic exponent of $\kappa$.
Let $f: C \to \Spec(R)$ be a  relative curve of relative dimension $1$
 and denote by $C^{sm}$ the smooth locus of $f$.
We assume that 
one of the following holds:

\smallskip

\noindent (I) $C^{sm}_\kappa$ is  dense in $C_\kappa$;

\smallskip

\noindent (II) $R$ is a DVR  and 
$C$ is integral regular.

\smallskip

\noindent Let $G$ be a reductive $C$--group scheme and consider its presentation
\cite[XXII.6.2.3]{SGA3}
$$
1 \to \mu \to G^{sc} \times_C \rad(G) \to G \to 1,
$$
where $\rad(G)$ is the radical $C$-torus of $G$ and
$G^{sc}$ is the simply connected universal cover of $DG$.
We assume that 

\smallskip

(i) $\mu$ is \'etale over $C$;

\smallskip

(ii)  the $C$--torus $\rad(G)$ is 
is quasi-split by a a finite \'etale extension $\widetilde{C}/C$ of degree prime to $p$.

Let $E$, $E'$  be  two $G$-torsors over $C$ such that 
$E \times_{C} C_\kappa$ is isomorphic
to $E' \times_C C_\kappa$. Then $E$ and $E'$ are  locally isomorphic
for the Zariski topology.
\end{stheorem}

 \begin{proof}
 Once again we can assume that $E'=G$.
We consider the following commutative diagram
\[\xymatrix@1{
H^1(C,\mu) \ar[r] \ar[d] &  H^1(C,G^{sc}) \times H^1(C,T) \ar[r] \ar[d] & 
H^1(C,G)  \ar[r] \ar[d] & H^2 (C,\mu)  \ar[d] \\
H^1(C_\kappa,\mu) \ar[r]  &  H^1(C_\kappa,G^{sc})
\times H^1(C_\kappa,T) \ar[r]  & 
H^1(C_\kappa,G)  \ar[r]  & H^2 (C_\kappa,\mu) ,
}\]
  where the horizontal lines are exact sequences of pointed sets.
  On the other hand, the proper base change theorem for 
  \'etale cohomology shows that the maps $H^i(C,\mu) \to H^i(C_\kappa,\mu)$
  are bijective for $i=1,2$ \cite[XII.5.5.(iii)]{SGA4}.
  By diagram chase, it follows that the map
  $$
  \ker\bigl( H^1(C,G^{sc}) \to H^1(C_\kappa,G^{sc}) \bigr)
  \times  \ker\bigl( H^1(C,T) \to H^1(C_\kappa,T_\kappa) \bigr) \qquad
  $$
  $$
  \qquad \to   \ker\bigl( H^1(C,G) \to H^1(C_\kappa,G) \bigr)
  $$
  is onto. The first kernel (resp.\ the second one) 
    consists of Zariski locally trivial
  according to Theorem \ref{thm_main} (resp.\ 
  Proposition  \ref{prop_MB}.(3)
  and Theorem \ref{thm_torus}).
  Thus  the third  kernel  consists of  Zariski locally trivial torsors.
  \end{proof}

%\begin{sremark}{\rm 
%The main argument used there is the stratification of
%Conrad-Lieblich-Olsson \cite{CLO} which refines that of Gruson-Raynaud.
%}
%\end{sremark}

%Replaced by a simpler argument provided by Moret-Bailly.

We have the following refinement of Theorem \ref{thm_red} which answers  
a question of Olivier Benoist.

\begin{stheorem}\label{thm_ref} Assume that $R$ 
is  local henselian noetherian of  residue field $\kappa$.
 Let $f: C \to \Spec(R)$ be a  smooth  
 relative curve of relative dimension $1$
 Let $G$ be a reductive $C$--group scheme which satisfies the same assumption as in Theorem \ref{thm_red}.
 
Let $E,E'$ be  $G$-torsors over $C$. Then the following are equivalent:

\smallskip

(i) The $G$--torsors $E,E'$ are locally isomorphic  for the Zariski topology; 
\smallskip

(ii) The $G_{C_\kappa}$--torsors $E_{C_\kappa}$,  $E'_{C_\kappa}$ are locally isomorphic for the Zariski topology.

\end{stheorem}

The proof of Theorem \ref{thm_ref} involves the Iwasawa decomposition; since  
 the reference  \cite[cor. 7.3.2.(ii)]{BrT} is for
 the semisimple case, we provide a short proof for the reductive case.

\begin{slemma} \label{lem_iwasawa} Let $O$ be a henselian DVR of fraction field 
$K$. Let $H$ be an  $O$--reductive group scheme and 
$P$ a minimal $O$--parabolic subgroup of $H$. Let $P=U \rtimes L$
be a Levi decomposition (it exists according to \cite[XXVI.2.3]{SGA3})
and let $S$ be the maximal split central $O$-torus of $L$. Then we have an isomorphism  
$$
S(K)/S(O) \simlgr  U(K) \backslash H(K)/ H(O) .
$$ 
\end{slemma}

\begin{proof} Let $\gX \cong  H/P$ be the $O$-scheme of parabolic 
subgroups of $H$ with same type as $P$ \cite[\S XXVI.3]{SGA3}.
Since $\gX$ is a projective $O$-scheme, we have
$\gX(O)=\gX(K)$. Since $H(O)$ (resp.\ $H(K)$) acts
simply transitively on $\gX(O)$ (resp.\ $\gX(K)$)
according to \cite[XXVI.2.5]{SGA3},
we get that $ H(O) /P(O)=H(K)/P(K)$.
It follows that   \begin{equation}\label{id_iwa}
H(K)=P(K)H(O)= U(K)\, L(K) \, H(O).
\end{equation}

Since $L/S$ is $O$--anisotropic, it is $K$--anisotropic
\cite[3.4]{Gu}.
We have then 
$(L/S)(O)=(L/S)(K)$ according to the Bruhat-Tits-Rousseau's theorem \cite[3.5]{Gu}.
 Hilbert 90 theorem states that  $H^1(O,S)= H^1( K,S)=0$
hence  $L(O)/S(O)=L(K)/S(K)$. It follows that  
$L(K)=S(K)L(O)$; by taking into account the identity \eqref{id_iwa}
we obtain $H(K)= U(K) \, S(K) \,  H(O)$.
We have proven that the map
$$
S(K)/S(O) \to U(K) \backslash H(K)/ H(O) 
$$
is onto. For establishing the injectivity we are given
$s, s' \in S(K)$ such that $s'= u s h$ with $u \in H(K)$, $h \in H(O)$.
We consider the $O$-subgroup scheme $M= U \rtimes S$ of
$P$. Then $h=u_0 \, s_0 \in M(O)= U(O) \rtimes S(O) $ so that
$s'= u \,  s \, u_0  \, s_0 $. We conclude that 
 $s'=  s \,  s_0$.
\end{proof} 
 
We can now proceed to the proof of Theorem 
\ref{thm_ref}.

 \begin{proof}
 Once again we can assume that $E'=G$.
  The implication $(i) \Longrightarrow (ii)$ is obvious. 
Conversely we assume that $E_{C_\kappa}$ is locally trivial for the Zariski topology.
 Then there exists  a divisor $D_\kappa=\{c_1, \dots , c_t\}$  of $C_\kappa$
 such that the $G$--torsors $E_{V^D_\kappa}$ and $E_{D_\kappa}$  are trivial.
 We  denote by 
 $\kappa_i=\kappa(c_i)$ the residue field, it is finite separable over $\kappa$.
 Let $R_i$ be the finite \'etale cover of $R$
 which lifts $\kappa_i/\kappa$; 
 Hensel's lemma applies to $C \times_R R_i$ 
shows that each $c_i$ 
lifts in a closed $R$--subscheme
$D_i \to C$ which is finite \'etale over $R$.
We put $D=D_1 \sqcup D_2 \dots \sqcup D_t$, 
choose $V=\Spec(A)$ containing $D$,  as  in  \S \ref{section_unif}.

 \begin{sclaim}\label{claim_extend} The $G$--torsor $E_{C_\kappa}$ extends  to  a 
 $G$--torsor $F$ whose restrictions to $V^D$ and $D$ are trivial. 
 \end{sclaim}

 We postpone the proof of the Claim. Assuming the Claim, the $G$--torsors
 $E$ and $F$ are isomorphic on $C_\kappa$. Theorem \ref{thm_red}
 shows that $E$ and $F$ are locally isomorphic for the Zariski topology.
 It follows that $E_{V^D}$ is  locally trivial.  By varying  the choices  of $D$,
 we can find a cover of $C$ by open subsets $V'_1, \dots , V'_r$
 such that $E_{V'_i}$ is  locally trivial for $i=1,\dots, r$.
 Thus $E$ is locally trivial 
 for the Zariski topology.
For proving the Claim, we use the uniformization bijections
\[\xymatrix@1{
G(V^D) \backslash L_G(R) /  L^+_G(R)
 \ar[d] \ar[r]^{\sim \qquad \qquad } &
\mathrm{ker}\Bigl(  H^1(C,G) \to H^1(V^D,G) \times
H^1\bigl( \widehat A, G \bigr) \Bigr) \ar[d] \\
G(V^D_\kappa) \backslash L_G(\kappa) /  L^+_G(\kappa)
 \ar[r]^{\sim \qquad \qquad } & \ker\Bigl( 
H^1(C_\kappa,G) \to H^1(V^D_\kappa,G) \times
H^1\bigl( \widehat{A_\kappa}, G \bigr) \Bigr). 
}\]
using the notations of \S \ref{section_unif}.
The $G_{C_\kappa}$--torsor $E_\kappa$ arises then from 
an element $g \in L_G(\kappa)$.
We shall  show that the map $L_G(R) \to L_G(\kappa) /  L^+_G(\kappa)$ is onto, that is, 
$$
\phi:  G\bigl( \widehat A \otimes_A A_\sharp\bigr) \to
G\bigl( \widehat{A_\kappa} \otimes_A A_\sharp \bigr) /  
G\bigl( \widehat{A_{\kappa}} \bigr)
$$
is onto (which implies Claim \ref{claim_extend}).
According to \cite[24.19]{Wa}, we have a decomposition 
 $\widehat A = {\widehat A}_1 \times \cdots \times {\widehat A}_t$ in complete local rings
 where $ {\widehat A}_i$ has residue field   $\kappa_i$. 
We have a compatible  decomposition $\widehat{A_\kappa}  = 
\cA_1 \times \cdots  \times \cA_t$ 
where $\cA_i$ is a complete DVR of residue field $\kappa_i$.
Since the map $\widehat{A} \to  \widehat{A_\kappa}$
is onto (by the Mittag-Leffler's condition), it follows that each map 
${\widehat A}_i \to \cA_i$ is onto.
We are then reduced to show that each map 
$$
\phi_i:  G\bigl( \widehat A_i \otimes_A A_\sharp\bigr) \to G\bigl( \cA_i  \otimes_{A} A_\sharp \bigr) /  
G\bigl( \cA_i \bigr)
$$
is onto. We fix an index $i$.
Let $\cF_i$ be the fraction field of  $\cA_i$, we have $\cA_i  \otimes_{A} A_\sharp = \cF_i$.
The quotient $G\bigl( \cF_i \bigr) / G\bigl( \cA_i \bigr)$ is described by the Iwasawa decomposition.

Let $Q_i$ be a minimal parabolic $\kappa_i$--subgroup of $G_{\kappa_i}$, it lifts to 
an ${\widehat A}_i$--parabolic $P_i$ subgroup of $G_{{\widehat A}_i}$.
We have $P_i =U_i \rtimes L_i$ where 
 $U_i$ is the unipotent radical of $P_i$  and $L_i$ is a Levi subgroup.
 Let $S_i \subset L_i$ be the maximal ${\widehat A}_i$-split central torus of
\cite[XXVI.7.8]{SGA3}; then 
$S_{i, \kappa_i}$ is a maximal $\kappa_i$--torus of $G_{\kappa_i}$ and 
$S_{i, \cF_i}$ is a maximal $\cF_i$--torus of $G_{\cF_i}$.
We write the  Iwasawa decomposition (Lemma \ref{lem_iwasawa})
$$
  S_i(\cF_i)/  S_i(\cA_i) \simlgr  U_i(\cF_i)  \backslash G\bigl( \cF_i \bigr) / G\bigl( \cA_i \bigr).
$$
In particular, $ U_i(\cF_i) \,  S_i(\cF_i)$ maps onto $G\bigl( \cF_i \bigr) / G\bigl( \cA_i \bigr)$.
We consider the commutative diagram
\[\xymatrix@1{
U_i( \widehat{A}_i  \otimes_A A_\sharp ) \,
S_i( \widehat{A}_i  \otimes_A A_\sharp ) \ar[r] \ar[d] &  U_i(\cF_i) \, 
S_i(\cF_i)  \ar@{->>}[d] \\
G( \widehat{A}_i  \otimes_A A_\sharp ) / G( \widehat{A}_i)
  \ar[r]^{\phi_i} & 
G\bigl( \cF_i \bigr) / G\bigl( \cA_i \bigr).
}\]
The surjectivity of $\phi_i$  follows of the next

\begin{sclaim} \label{claim_iwa} The maps $U_i( \widehat{A}_i  \otimes_A A_\sharp ) \to
U_i(\cF_i)$ and $S_i( \widehat{A}_i  \otimes_A A_\sharp ) \to
S_i(\cF_i)$ are onto.
 
\end{sclaim}

 Since the map $\widehat{A}_i \to  \cA_i$
is onto  it follows that 
the map $\widehat{A}_i  \otimes_A A_\sharp \to  \cA_i \otimes_{A_\kappa} 
A_{\sharp}=\cF_i$ is onto. According to \cite[XXVI.1.12]{SGA3},
$U_i$ is isomorphic (as $\widehat{A}_i$-scheme) to a vector group scheme so that 
$U_i( \widehat{A}_i  \otimes_A A_\sharp ) \to
U_i(\cF_i)$ is onto. 

Since $S_i$ is a split torus, it is enough 
 to check the surjectivity of 
$( \widehat{A}_i  \otimes_A A_\sharp )^\times  \to
\cF_i^\times$. 
Since $\widehat{A}_i$ is a regular local ring, 
the divisor $D_{\widehat{A}_i}$ is principal, i.e. $D_{\widehat{A}_i}= \mathrm{div}(\pi_i)$
with  $\pi_i \in \widehat{A}_i$.
It follows that $\widehat{A}_i  \otimes_A A_\sharp= \widehat{A}_i\bigl[  \frac{1}{\pi_i} \bigr]$
and $\cF_i=\cA_i \bigl[  \frac{1}{\ol{\pi}_i}\bigr]$ where 
$\ol{\pi}_i$ is the image of $\pi_i$ in $\cA_i$.

An element of $\cF_i^\times$ is of the shape $\ol{\pi}_i^r 
\bigl(a_0+ a_1  \ol{\pi}_i  + a_2\ol{\pi}_i^2+  \dots \bigr)$
with $a_0 \in \cA_i^\times$, $a_1,a_2, \dots \in \cA_i$. It is lifted to 
$\Bigl( \widehat{A}_i\bigl[  \frac{1}{\pi_i} \bigr] \Bigr)^\times$ 
by the element $\pi_i^r 
\bigl( \tilde{a_0}+ \widetilde{a_1}  \pi_i  + \widetilde{a_2} \pi_i^2+  \dots \bigr)$
where the  $\widetilde{a_i}$'s lift the $a_i$'s.
 Claim \ref{claim_iwa} is established and ends the proof.
 \end{proof}

 \begin{sremarks}{\rm (a) In case (II), Claim \ref{claim_extend} holds 
 provided $D_\kappa$ is contained in the smooth locus 
 of $C_\kappa$.
 
 \smallskip
 
 \noindent (b) In the case $G$ is semisimple simply connected and $k$ is algebraically closed
 it is well-known that $G$--torsors over $C_\kappa$ are 
 locally trivial for the Zariski topology. Theorem \ref{thm_ref} \, 
 provides then an alternative proof of Drinfeld-Simpson's theorem
 in this case.
  }
 \end{sremarks}

\begin{scorollary} \label{cor_ref0} Assume that $R$ 
is  a, henselian DVR with finite residue field $\kappa$.
 Let $f: C \to \Spec(R)$ be a  smooth  proper curve
 Let $G$ be  a semisimple simply connected $C$--group scheme.
 Then $H^1_{Zar}(C,G)=H^1(C,G)$. 
\end{scorollary}

\begin{proof} By a theorem of   Harder \cite{Hd2}, 
we have $H^1(\kappa(C^\theta_\kappa),G)=1$ for each connected component $C_\kappa^\theta$ of
$C_\kappa$. Nisnevich's theorem
\cite{N} (see also \cite{Gu})   shows  that    $H^1_{Zar}(C_\kappa,G)=H^1(C_\kappa,G)$.
Thus the corollary   follows from Theorem \ref{thm_ref}.
\end{proof}

\begin{scorollary} \label{cor_ref} Assume that $R$ 
is  an henselian DVR with   residue field $\kappa$.
 Let $f: C \to \Spec(R)$ be a  smooth  projective curve such that its 
 generic fiber is connected.
 Let $G$ be a reductive $C$--group scheme which satisfies the same assumption as in Theorem \ref{thm_red}.
 Let $F$ be the function field of $C$. Then, 
the local-global principle holds for $G$-torsors over $F$ with
respect to all discrete valuations arising from 
codimension one points of $C$.
\end{scorollary}

\begin{proof} Let $\xi \in H^1(F, G)$ which is trivial 
 over $F_v$ for all completions at discrete valuations 
 arising from codimension one points of $C$. By glueing \cite[cor. A.8]{GP}, 
there is an element $\zeta \in  H^1(U, G)$ 
which maps to  $\xi$ over $F$ where $U \subset C$
contains all points of codimension $1$. 
According to \cite[th.\ 6.13]{CTS0}, we have 
$H^1(C, G) =H^1(U,G)$ so that we can assume that $U=X$. 
Since $\xi$ is
trivial over the completion $F_w$ of $F$ at the discrete 
valuation $w$ associated to the special fiber of $C$, 
we claim that the 
specialisation $\zeta_\kappa \in H^1(C_\kappa ,G)$ of $\zeta$
is generically trivial. This follows from the
fact that $H^1(O_w, G)$ injects in $H^1(F_w, G)$
due to Bruhat-Tits (see \cite[Th. 5.1]{Gu}).
According to  Nisnevich's theorem
\cite{N} (see also \cite{Gu}) this class is Zariski locally trivial on $C_\kappa$.
Theorem \ref{thm_ref} enables us to conclude that 
$\eta$ is Zariski locally trivial and hence $\xi$ is trivial.
\end{proof}

\begin{sremarks}{\rm
(a) 
The only  case where we knew that local-global
principle for $G$ simply connected group defined over $C$ (for arbitrary residue fields)
is when  $G_{F}$ is $F$-rational.
In this case, Harbater, Hartmann and Krashen 
established  their ``patching local-global principle''
\cite[th. 3.7]{HHK} which implies our local-global principle
according to \cite[Th. 4.2.(ii)]{CTPS}.

\smallskip

\noindent (b) 
The  special case when  $R$ is the ring of 
integers of a $p$-adic field was already known
\cite[Th. 4.8]{CTPS}. 
}
\end{sremarks}

\medskip

\section{Appendices} \label{section_appendix}

The purpose of this appendix is to provide proofs  to statements  
for algebraic spaces and stacks
which are well-known among  experts.

\subsection{Jacobian criterion  for stacks}

Let $S$ be a scheme and  let $\cX$, $\cY$
be quasi-separated algebraic $S$--stacks of finite presentation.
Let  $g: \cX \to \cY$ be a $1$-morphism over $S$. 
We have a $1$--morphism $Tg: T(\cX) \to T(\cY)$ of
algebraic stacks \cite[17.14,  17.16]{LMB}.

Let $s  \in   S$ and denote by $K$ the 
residue field of $s$. Let $x:
\Spec(K) \to \cX$ be a $1$-morphism mapping to $s$. 
We put  $T(\cX)_x= T(\cX/S) \times_{\cX} \Spec(K)$
and denote by  $\mathrm{Tan}_x(\cX)$ the category  $T(\cX)_x(K)$.
We denote by $y= g \circ x: \Spec(K) \to \cY$ and get 
the tangent morphism 
$(Tg)_x: \mathrm{Tan}_x(\cX) \to \mathrm{Tan}_y(\cY)$.

\newpage

\begin{sproposition} \label{prop_jacobian}
We assume that $\cX$ is smooth at $x$ over $S$. Then the
following assertions are equivalent:

\smallskip 

\noindent (i) The morphism $g$ is smooth at $x$;

\smallskip

\noindent (ii)  The tangent morphism $(Tg)_x:
\mathrm{Tan}_x(\cX) \to \mathrm{Tan}_y(\cY)$
is essentially surjective.

\smallskip

\noindent Furthermore, under those conditions, $\cY$ is smooth at $y$ over $S$.
\end{sproposition}

 \begin{proof}
 In the case of 
  a morphism $g: X \to Y$ of $S$--schemes locally of finite presentation
  such that $g(x)=y$ and $X$ is smooth at $x$ over $S$,
  we have that $K=\kappa(x)=\kappa(y)$ so that the statement
  is a special case of \cite[17.11.1]{EGA4}.
   We proceed now to the stack case.

\smallskip

\noindent $(i) \Longrightarrow (ii).$  Up to shrinking, we can assume
 that $\cX$ is smooth over $S$ and that $g$ is smooth.
 We denote by $i: \Spec(K) \to \Spec(K[\epsilon])$.

 We are given an object of  $\mathrm{Tan}_y(\cY)$, that is 
 a morphism couple $(y', \eta)$ where $y': \Spec(K[\epsilon]) \to \cY$
 together with a $2$--morphism $\eta:   y'  \circ i \to y$.
 We remind that $g$ is formally smooth \cite[Tag 0DP0]{St} that is,   if it is formally smooth on objects as
 a $1$-morphism in categories fibered in groupoids \cite[Tag 0DNV]{St}.
A special case is  for  the following commutative  diagram
\[\xymatrix@1{
\Spec(K) \ar[d]^i \ar[r]^{\quad x} & \gX \ar[d]^g  \\
\Spec(K[\epsilon] ) \ar[r]^{\quad y'} \ar@{.>}[ru] & \gY,
}\]
where $\eta: y' \circ i \to y= g \circ x=y$ is a  $2$-morphism witnessing the commutativity
 of the diagram;  there exists  a triple $(x', \alpha, \beta)$ where :

 \smallskip

\noindent (i) $x': \Spec(K[\epsilon] ) \to \gX$ is a morphism;

 \smallskip

 \noindent (ii) $\alpha: x' \circ i \to x$, $\beta: y' \to g \circ x'$ are $2$--arrows such that 
  $\eta= ( id_g \star \alpha ) \circ (  \beta \star id_i )$.
  
  \smallskip
  
 \noindent  It follows that $g(x', \alpha)= (g \circ x', g \circ  \alpha) \in 
 \mathrm{Tan}_y(\cY)$ is isomorphic to $(y', \eta)$.
 This establishes the  essential surjectivity
 of the tangent morphism.

 \smallskip
 
 \noindent $(ii) \Longrightarrow (i).$ According to \cite[Thm. 6.3]{LMB}, there
 exists a smooth $1$--morphism $\varphi : Y \to \cY$ and a point $y_1 \in Y(K)$
 mapping to $y$ such that $Y$ is an affine scheme.
 We note that $K=\kappa(y_1)$.
 We consider the fiber product $\cX'= \cX \times_\cY Y$, it is an algebraic
 stack and there exists a $1$--morphism $x': \Spec(K) \to \cX'$ lifting  $x$ and $y_1$.
 There exists a smooth $1$--morphism $\psi : X' \to \cX'$ and a point $x_1 \in X'(K)$
 mapping to $x$ such that $X'$ is an affine scheme. By construction
 we have again that $K=\kappa(x_1)$. We have then 
  the commutative diagram 
  
\[\xymatrix@1{
X' \ar[r]^\psi  & \cX' \ar[d] \ar[r]^{g'} & Y  \ar[d]^\varphi\\
  & \cX \ar[r]^g  & \cY .
}\]
 
 \noindent According to \cite[Lem.\ 17.5.1]{LMB}, the square 
 
\[\xymatrix@1{
  T(\cX'/S) \ar[d] \ar[r]^{Tg'} & T( Y/S) \ar[d]^{T\varphi} \\
 T(\cX/S) \ar[r]^{Tg}  & T(\cY/S) 
}\]
is $2$--cartesian. It follows that the square

\[\xymatrix@1{
  \mathrm{Tan}_{x'}(\cX') \ar[d]^{(T\psi)_{x_1}} \ar[rr]^{(Tg')_{x'}} & &
  \mathrm{Tan}_{y_1}(Y) \ar[d]^{(T\varphi)_{y_1}} \\
  \mathrm{Tan}_{x}(\cX)  \ar[rr]^{(Tg)_x}  && \mathrm{Tan}_{y}(\cY)
}\]
is $2$-cartesian.
Our assumption is that  the bottom morphism is essentially surjective,
it follows that $(Tg')_{x'}: \mathrm{Tan}_{x_1}(\cX') \to \mathrm{Tan}_{y_1}(Y)$
is essentially surjective as well.
Since $\psi$ is smooth, the map $(T\psi)_{x_1}: 
\mathrm{Tan}_{x'}(X') \to \mathrm{Tan}_{x_1}(\cX')$ is essentially surjective.
By composition it follows that $\mathrm{Tan}_{x_1}(X') \to \mathrm{Tan}_{y_1}(Y)$
is essentially surjective. Since $X'$ and $Y$ are locally of finite presentation
over $S$, the case of  schemes yields that $g' \circ \psi : X' \to Y$ is smooth at $x'$.
By definition of  smoothness for morphisms of  stacks \cite[\S 8.2]{O},
we conclude that $g$ is smooth at $x$.

\smallskip

We assume (ii) and shall show that $\cY$ is smooth at $y$ over $S$. 
Using the diagrams of the proof, we have seen that the $S$--morphism
 $X' \to Y$ of schemes is smooth at $x'$.
 Once again the classical Jacobian criterion \cite[17.11.1]{EGA4}
 applies and shows that $Y$ is smooth at $y_1$ over $S$.
 By definition of smoothness for stacks, we get that $\cY$ is smooth 
 at $y$ over $S$.
\end{proof}

\subsection{Lie algebra of an $S$-group space}\label{app_lie}

 Let $S$ be a scheme.
Let $f:  X \to Y$ be a morphism of $S$--algebraic spaces.
We consider the quasi-coherent sheaf $\Omega^1_{ X/Y}$ 
on $X$ defined in \cite[Tag 04CT]{St}.
Let $T$ be an $S$--scheme  equipped with
a closed subscheme $T_0$ defined by a quasi-coherent ideal
$\cI$ such that $\cI^2=0$. 
According to \cite[7.A page 167]{O} for any commutative diagram 
of algebraic spaces 
\[\xymatrix@1{
T_0 \ar[r]^{x_0}  \ar@{^{(}->}[d] & X \ar[d]^f \\
 T \ar[r]^y \ar@{.>}[ru] & Y .
}\]
if there exists a dotted arrow filling in the diagram then the set of such dotted
arrows form a torsor under $\Hom_{\cO_{T_0}}(x_{0}^*\Omega^1_{X/Y} , \cI)$.
We extend to group spaces well-known statements on   
group schemes  \cite[II.4.11.3]{SGA3}.

\begin{slemma}\label{lem_lie1}
Let $G$ be an $S$-group space. We denote by $e_G: S \to G$ the unit point 
and put $\omega_{G/S}=e_G^*\bigl( \Omega^1_{G/S})$.

\smallskip

\noindent (1) There is a canonical isomorphism  of $S$--functors
$\Lie(G) \simlgr \VV(\omega_{G/S})$ which is compatible
with the $\cO_S$--structure.
 
 \smallskip
 
 \noindent (2) If $\omega_{G/S}$ is a locally free coherent sheaf, 
 then $\Lie(G) \simlgr \WW(\omega_{G/S}^\vee)$.
 In particular we have an isomorphism 
 $$
 \Lie(G)(R) \otimes_R R' \simlgr \Lie(G)(R')
 $$
 for each morphism of $S$--algebras $R \to R'$.
 
 \smallskip
 
 \noindent (3) Assume that $G$ is smooth and quasi-separated over $S$.
 Then $\omega_{G/S}$ is a  finite locally free coherent sheaf and
 (2) holds.

\end{slemma}

Under the conditions of (2) or (3), we 
 denote also by $\cLie(G)=\omega_{G/S}^\vee$
 the locally free coherent sheaf.

\begin{proof}
(1) Let $T_0$ be an $S$-scheme and consider $T=T_0[\epsilon]$.
We apply the above fact
to the morphism $G \to S$ and the points $x_0=e_{G_{T_0}}$
and $y :T \to S$ the structural morphism.
It follows that $\ker\bigl( G(T) \to G(T_0) \bigr)$
is a torsor under $\Hom_{\cO_{T_0}}(e_{G_{T_0}}^* \Omega^1_{G/S} 
, \epsilon \, \cO_{T_0}) \cong 
\Hom_{\cO_{T_0}}(e_{G_{T_0}}^*\Omega^1_{G/S},\cO_{T_0})
= \Hom_{\cO_{T_0}}(\omega^1_{G/S} \otimes_{\cO_S} \cO_{T_0},\cO_{T_0})$.
We have constructed a isomorphism of $S$--functors
$\Lie(G) \, \simlgr \,  \VV(\omega_{G/S})$ and  the compatibility 
of $\cO_S$--structures  is a straightforward checking.

\smallskip

\noindent (2) If $\omega_{G/S}$ is a locally free coherent sheaf, 
 then $\Lie(G) \simlgr \VV(\omega_{G/S}) \simlgr \WW(\omega_{G/S}^\vee)$.  
The next fact follows from \cite[12.2.3]{EGA3}.

\smallskip

\noindent (3) According to \cite[Tag 0CK5]{St},
$\Omega^1_{G/S}$ is a finite locally free coherent sheaf over $G$.
If follows that $\omega^1_{G/S}$ is a finite locally free coherent sheaf over $S$.

\end{proof}

\begin{slemma}\label{lem_lie2} Let $G$ be a smooth $S$-group space and let 
$T$ be an $S$--scheme  equipped with
a closed subscheme $T_0$ defined by a quasi-coherent ideal
$\cI$ such that $\cI^2=0$. 
We denote by $t_0: T_0 \to S$ the structural 
morphism,  $G_0 = G \times_S T_0$ and assume that 
$t_0$ is  quasi-compact  and quasi-separated.

\smallskip

\noindent (1) We have an exact sequence of fppf (resp.\ \'etale, Zariski)
sheaves on $S$
$$
0 \to \WW\Bigl( (t_0)_*\bigl( \cLie(G_0) \otimes_{\cO_{T_0}} \cI \bigr) \Bigr)
\to \prod_{T/S} G \to \prod_{T_0/S} G \to 1  .
$$

\smallskip

\noindent (2) If $T= \Spec(A)$ is affine and 
$T_0=\Spec(A/I)$, we have  an exact sequence
$$
0 \to \Lie(G)(A) \otimes_A I \to G(A) \to G(A/I) \to 1.
$$
\end{slemma}

\begin{proof}
(1)  We  have
 $$\Hom_{\cO_{T_0}}(\omega_{G_0/T_0},\cI)
= H^0(T_0, \cLie(G_0) \otimes_{\cO_{T_0}} \cI)
= H^0\Bigl(T, t_{0,*}( \cLie(G_0) \otimes_{\cO_{T_0}} \cI) \Bigr), $$
whence  an exact sequence
$$
0 \to H^0\Bigl(T, t_{0,*}( \cLie(G_0) \otimes_{\cO_{T_0}} \cI) \Bigr) \to G(T) \to G(T_0).
$$
Now let $h: S' \to S$ be a flat morphism locally of finite presentation
and denote by $G'=G \times_S S'$, $h_T:  T' \to T$, $\dots$ the
relevant base change to $S'$.
Since $t_0$ is quasi-compact and quasi-separated, 
the flatness of $h$ yields an isomorphism \cite[Tag 02KH]{St}
$$
h^*\bigl( t_{0,*}( \cLie(G_0) \otimes_{\cO_{T_0}} \cI)\bigr)
\simlgr
t'_{0,*}\Bigl( h_{T_0}^*\bigl(  \cLie(G_0) \otimes_{\cO_{T_0}} \cI \bigr) \Bigr)
= t'_{0,*}( \cLie(G'_0) \otimes_{\cO_{T'_0}} \cI').
$$
The similar sequence for $T'_0$ reads then 
 $$
0 \to H^0\Bigl(T', t_{0,*}( \cLie(G_0) \otimes_{\cO_{T_0}} \cI) \Bigr)
\to G(T') \to G(T'_0).
$$
We have then an exact sequence of fppf sheaves
$$
0 \to \WW\Bigl( (t_0)_*\bigl( \cLie(G_0) \otimes_{\cO_{T_0}} \cI \bigr) \Bigr)
\to \prod_{T/S} G \to \prod_{T_0/S} G .
$$
For $T$ an affine scheme, the map $G(T) \to G(T_0)$
is onto since the smooth $S$--group space $G$ is formally smooth  
\cite[Tag 04AM]{St}, that is, it satisfies the infinitesimal lifting criterion  \cite[Tag 049S, 060G]{St}.

It implies the exactness for the the Zariski, \'etale and fppf
topologies.

\smallskip

\noindent (2) We can assume that $S=T=\Spec(A)$.
In this case, we have
$$H^0\Bigl(S, (t_0)_*\bigl( \cLie(G_0) \otimes_{\cO_{T_0}} \cI \bigr)\Bigr)
= H^0(T_0, \cLie(G_0) \otimes_{\cO_{T_0}} \cI)
= \Lie(G_0)(A/I) \otimes_{A/I} I.
$$
We have  $\Lie(G_0)(A/I) = \Lie(G)(A) \otimes_{A} A/I$ in view of Lemma 
\ref{lem_lie2}.(2) whence \break the identification
$\Lie(G_0)(A/I) \otimes_{A/I} I \cong  \Lie(G)(A) \otimes_{A} I$.
Then  (1) provides the exact sequence
$$
0 \to \Lie(G)(A) \otimes_A I \to G(A) \to G(A/I) 
$$
and the right map is onto since $G$ is smooth.
\end{proof}

 \begin{sremarks}\label{rem_lie2}{\rm 
  
  \noindent
  (a) A special case of (1) is $T=S[\epsilon]$ and $T_0=S$.
  We get an exact sequence of  fppf (resp.\ \'etale, Zariski)
sheaves on $S$
$$
0 \to  \WW( \cLie(G)) 
\to \prod_{S[\epsilon]/S} G \to  G \to 1  .
$$

  \smallskip
  
(b)  In  the group scheme case, (2)  is established
in \cite[proof of II.5.2.8]{DG}.

}
  
 \end{sremarks}

\bigskip

\medskip

\end{document}